\documentclass[12pt]{amsart}
\usepackage{graphicx}
\usepackage{color}
\usepackage[T1]{fontenc}
\usepackage[utf8]{inputenc}
\usepackage{tikz}
\usetikzlibrary[calc,patterns] 
\usepackage{pinlabel}

\theoremstyle{definition} 
\newtheorem{definition}{Definition}[section]

\theoremstyle{plain} 
\newtheorem{proposition}[definition]{Proposition}
\newtheorem{lemma}[definition]{Lemma}
\newtheorem{theorem}[definition]{Theorem} 
 
\newtheorem{convention}[definition]{Convention}
\newtheorem{corollary}[definition]{Corollary}

\theoremstyle{remark} 
\newtheorem{remark}[definition]{Remark}

\newtheorem*{Reader's guide}{Reader's guide}

\title[marked horizontal separatrices]{Moduli space of meromorphic differentials with marked horizontal separatrices}
\author{Corentin Boissy}
\address{
Université de Toulouse, CNRS, Institut Mathématique de Toulouse, UMR5219,  UPS IMT, F-31062 Toulouse Cedex 9, France}
\email{corentin.boissy@math.univ-toulouse.fr}


\date{\today}

\begin{document}
\begin{abstract}
We study framed translation surfaces corresponding to meromorphic differentials on compact Riemann surfaces, for which a horizontal separatrix is marked for each pole or zero. Such geometric structures naturally appear when studying flat geometry surfaces ``near'' the Deligne-Mumford  boundary.

We compute the number of connected components of the corresponding strata, and give a simple topological invariant that distinguishes them. In particular we see that for $g>0$, there are at most two such components, except in the hyperelliptic case.
\end{abstract}

\maketitle

\section{Introduction}

A nonzero holomorphic one-form ({Abelian differential}) on a compact Riemann surface naturally defines a flat metric with conical singularities on this surface. Geometry and dynamics on such flat surfaces, in relation to geometry and dynamics on the corresponding moduli space of Abelian differentials is a very rich topic and has been widely studied in the last 30 years. It is related to interval exchange transformations, billards in polygons, Teichmüller dynamics. 

A non-compact translation surface corresponds to a one-form on a non-compact Riemann surface. The dynamics and geometry on some special cases of non-compact translation surfaces have been studied more recently. 


In \cite{B:merom}, we have investigated the case of translation surfaces that come from meromorphic differentials defined on compact Riemann surfaces. In this case, we obtain non-compact translation surfaces with infinite area.
Such structures naturally appear when studying compactifications of strata of the moduli space of Abelian differentials. For instance, Eskin, Kontsevich and Zorich  \cite{EKZ}, based on results of Rafi \cite{Rafi}, showed
 that when a sequence of Abelian differentials $(X_i,\omega_i)$ converges to a boundary point in the Deligne-Mumford compactification, then subsets $(Y_{i,j},\omega_{i,j})$ corresponding to thick components of the $X_i$, after suitable rescaling converge to meromorphic differentials (see \cite{EKZ}, Theorem~10). Similar results were independently proved by Grushevsky and Krichever \cite{GruKri}, by Koch and Hubbard \cite{KH} and by Smillie. See also \cite{BCG+}.

 In this paper, a meromorphic differential on a compact Riemann surface will be called \emph{translation surface with poles}, or simply translation surface when there is no confusion with the usual (compact) translation surfaces.
 
 This work was suggested to the author by Smillie, as a step in a project of constructing a geometric compactification of the strata of the moduli space of Abelian differentials by using only flat geometry. A (compact) translation surface ``near'' the boundary, should be seen as a collection of translation surfaces with poles, glued together suitably after cutting out a neighborhood of a collection of singularities (including all the poles, in order to obtain in the end a compact translation surface). However, the gluing operation requires some extra combinatorial data, that can be expressed in terms of a ``frame'' on the translation surfaces with poles.

As in \cite{B:labeled}, a framed translation surface is a translation surface with a choice, for each singularity of a horizontal separatrix (see Section~\ref{mod:space:framed} for a precise definition). When the singularity is a conical singularity (\emph{i.e.}, a zero of the corresponding one-form), it corresponds to a horizontal separatrix. When the singularity corresponds to a non-simple pole,  it corresponds to an equivalence class of horizontal geodesics going to infinity for the flat metric. A singularity of degree $n\in \mathbb{Z}$ will have $|n+1|$ possible choices of horizontal separatrices. Such a framed translation surface will be also called a translation surface with marked horizontal separatrices.

The number of connected components of the moduli space of framed (compact) translation surfaces was computed by the author in \cite{B:labeled}. In this paper, we answer the same question for the moduli space of framed translation surfaces with poles. 

The first theorem deals with the case of nonhyperelliptic connected components in genus at least 1. 
\begin{theorem}\label{MT:g}
Let $g\geq 1$. Let $\mathcal{H}$ be a stratum of the moduli space of genus $g$ meromorphic differentials, and $\mathcal{C}\subset \mathcal{H}$ be a nonhyperelliptic connected component. Let  $\mathcal{H}_\mathcal{C}^{hor}$ be the moduli space of translation surfaces in $\mathcal{C}$ with marked horizontal separatrices.  We assume that the set of poles does not consists of a pair of simple poles. We have:
\begin{itemize}
\item If there exists a simple pole, or if there are only even degree singularities, then $\mathcal{H}_\mathcal{C}^{hor}$ is connected.
\item Otherwise, $\mathcal{H}_\mathcal{C}^{hor}$ has two connected components that are distinguished by the invariant $Sp$ defined in Section~\ref{topo:inv}.
\end{itemize}

When the set of poles consists of a pair of simple poles, we have the following result.
\begin{itemize}
\item If there are only even degree zeroes, then $\mathcal{H}_\mathcal{C}^{hor}$ is connected.
\item Otherwise, $\mathcal{H}_\mathcal{C}^{hor}$ has two connected components that are distinguished by the invariant $Sp$ defined in Section~\ref{topo:inv}.\end{itemize}
\end{theorem}
The topological invariant $Sp$ that distinguishes the connected components $\mathcal{H}_\mathcal{C}^{hor}$ is a variation of the classical Arf invariant for moduli space of Abelian differentials, and is therefore easily computable in terms of the flat structure.

The case of hyperelliptic connected components is easy and studied in Section~\ref{hyp:case}. In this case, there are more connected components for $\mathcal{H}_\mathcal{C}^{hor}$ due to the extra symmetry of the surfaces.

The genus zero case is particular: there might be many more components, as described in the following theorem.
\begin{theorem}\label{MT:g0}
Let $\mathcal{H}=\mathcal{H}(n_1,\dots n_r)$ be a stratum of genus zero translation surfaces. Let  $\mathcal{H}^{hor}$ be the moduli space of translation surfaces in $\mathcal{H}$ with marked horizontal separatrices.
Let 
$$N=\prod_{i,j} \gcd\left(\{n_k\}_{k\notin \{i,j\}} \cup \{n_i+1,n_j+1\}\right)$$
\begin{itemize}
\item If there exists $i\in \{1,\dots ,r\}$ such that $n_i=-1$, then $\mathcal{H}^{hor}$ is connected.
\item If all $n_i$ are different from $-1$ and if there are at most two odd degree singularities, then there are $N$ connected components of $\mathcal{H}^{hor}$ that are distinguished by the invariant $\Phi$ defined in Section~\ref{zero:genus}.
\item Otherwise, there are $2N$ connected components  of $\mathcal{H}^{hor}$ that are distinguished by the invariant $(\Phi,Sp)$.
\end{itemize}
\end{theorem}
The topological invariant $\Phi$ in the above theorem is easily computable in terms of the flat structure. The idea is to look at indices of the Gauss map  modulo relevant integers for a certain collection of paths.

\subsection*{Structure of the paper}
The paper is organized as follows:
\begin{itemize}
\item Section~\ref{sec:flat:merom} is devoted to generalities and background about translation surfaces with poles. The classification theorem of the connected components of moduli space of meromorphic differentials by the author is recalled, and few important statements about the structure of these connected components. We end with the proof of a preliminary result about the existence, in each connected component, of a surface with a pole of prescribed degree and zero residue.
\item Section~\ref{mod:space:framed} gives the precise definition of the moduli space of framed meromorphic differentials, and reduces the problem to the computation of the index of a subgroup $H$ of a product of cyclic groups.
\item Section~\ref{elem:moves} describes paths in the underlying stratum that produces some particular elements  in $H$ that will be ultimately proven to be the  generators of $H$. One key step there is to show that these elements exist for each connected component of each stratum. 
\item Section~\ref{pos:genus} defines first a topological invariant for the positive genus case, then proves Theorem~\ref{MT:g}.
\item Section~\ref{zero:genus} defines a topological invariant for the zero genus case, then proves Theorem~\ref{MT:g0}.
\end{itemize}

\subsubsection*{Acknowledgements}
I thank John Smillie for motivating the work on this paper and interesting discussions. This work is partially supported by the ANR Project "GeoDym".

\section{Preliminaries}\label{sec:flat:merom}

\subsection{Holomorphic one-forms and flat structures}
Let $X$ be a Riemann surface and let $\omega$ be a holomorphic one-form. For each $z_0\in X$ such that $\omega(z_0)\neq 0$, integrating $\omega$ in a neighborhood of $z_0$ gives local coordinates whose corresponding transition functions are translations, and therefore $X$ inherits a flat metric, on $X\backslash \Sigma$, where $\Sigma$ is the set of zeroes of $\omega$.

In a neighborhood of an element of $\Sigma$, such metric admits a conical singularity of angle $(k+1)2\pi$, where $k$ is the order of the corresponding zero of $\omega$. Indeed, a zero of order $k$ is given locally, in suitable coordinates by $\omega=(k+1)z^k dz$. This form is precisely the pre-image of the constant form $dz$ by the ramified covering $z\to z^{k+1}$. In terms of the flat metric, it means that the flat metric defined locally by a zero of order $k$ appears as a connected covering of order $k+1$ over a flat disk, ramified at zero.

When $X$ is compact, the pair $(X,\omega)$, seen as a smooth surface with such translation atlas and conical singularities, is usually called a \emph{translation surface}.

If $\omega$ is a meromorphic differential on a compact Riemann surface $\overline{X}$, we can consider the translation atlas defined by $\omega$ on $X=\overline{X}\backslash \Sigma'$, where $\Sigma'$ is the set of poles of $\omega$. We obtain a translation surface with infinite area. We will call such a surface a \emph{translation surface with poles}, or simply a \emph{translation surface}. 

\begin{convention}
When speaking of a translation surface with poles $S=(X,\omega)$: the surface $S$ equipped with the flat metric is noncompact; the underlying Riemann surface $X$ is a punctured surface and $\omega$ is a holomorphic one-form on $X$; the corresponding closed Riemann surface is denoted by $\overline{X}$, and $\omega$ extends to a meromorphic differential on $\overline{X}$ whose set of poles is precisely $\overline{X}\backslash X$.
\end{convention}

As in the case of Abelian differentials, a \emph{saddle connection} is a geodesic segment that joins two conical singularities (or a conical singularity to itself) with no conical singularities on its interior. 

We fix some terminology, that we will use during this paper.
\begin{itemize}
\item The order, or degree of a zero of $\omega$ is defined as usual. The cone angle at a zero of degree $n$ is $2\pi(n+1)$.
\item The order of a pole of $\omega$ is defined as usual. It is a \emph{positive} integer.
\item A singularity of $(X,\omega)$ is a zero or a pole of $\omega$. By convention, the \emph{degree} of the singularity $P$ will correspond to  its order if $P$ is a zero, or the opposite of its order if $P$ is a pole. For instance, a pole of order 2  corresponds to a singularity of degree -2. We denote by $\deg(P)\in \mathbb{Z}$ the degree of $P$. 
\end{itemize}


With the above convention, we  recall that it is well known that $\sum_{i=1}^r n_i=2g-2$, where $\{n_1,\dots ,n_r\}$ is the set (with multiplicities) of the degree of the singularities of $(X,\omega)$.

\subsection{Local model for poles}
 The neighborhood of a pole in $\overline{X}$ of order one is an infinite cylinder with one end. Indeed, up to rescaling, the pole is given in local coordinates by $\omega=\frac{1}{z}dz$. Writing $z=e^{z'}$, we have $\omega=dz'$, and $z'$ is in an infinite cylinder.

Now we describe the flat metric in a neighborhood of a pole in $\overline{X}$ of order $k\geq 2$ (see also \cite{strebel, B:merom}).
First, consider the meromorphic 1-form on $\mathbb{C}\cup \{\infty \}$ defined on $\mathbb{C}$ by $\omega=z^k dz$. Changing coordinates $w=1/z$, we see that this form has a pole $P$ of order $k+2$ at $\infty $, with zero residue. In terms of the translation structure, a neighborhood of the pole is obtained by taking an infinite cone of angle $(k+1)2\pi$ and removing a compact neighborhood of the conical singularity. Since the residue is the only local invariant for a pole of order k, this gives a local model for a pole with zero residue.

Now, define $U_R=\{z\in \mathbb{C}| |z|>R\}$ equipped with the standard flat metric.
Let $V_R$  be the Riemann surface obtained after removing from $U_R$ the ${\pi}$--neighborhood of the real half line $\mathbb{R}^-$, and identifying by the translation $z\to z+\imath 2\pi$ the lines $-\imath {\pi}+\mathbb{R}^-$ and $\imath {\pi}+\mathbb{R}^-$. The surface $V_R$ is naturally equipped with a holomorphic one-form  $\omega$ coming from $dz$ on $V_R$. We will show that this one-form has a pole of order~2 at infinity and residue -1.
Start from the one-form on $U_{R'}$ defined by $(1+1/z)dz$ and integrate it. Choosing the usual determination of $\ln(z)$ on $\mathbb{C}\backslash \mathbb{R}^-$, 
one gets the map $z\to z+\ln(z)$ from $U_{R'}\backslash \mathbb{R}^-$ to $\mathbb{C}$, which extends to an injective holomorphic map $f$ from $U_{R'}$ to $V_R$, if $R'$ is large enough. We claim that $f$ is also surjective around infinity, \emph{i.e.} for $Z\in V_R$ with large enough modulus, there exists $z\in U_{R'}$ with $f(z)=Z$.  Indeed, considering $Z$ as an element of $\mathbb{C}$, we consider the map $g(z)=Z-\ln(z)$ which is contracting. Choosing $\rho=2\ln(|Z|)$ we consider the ball $B$ of center $Z$ and radius $\rho$. If $B$ intersects $\mathbb{R}^-$, we consider $B^+=B\cap\{Im(z)>0\}$ if $Im(Z)>0$ (or $B^-=B\cap\{Im(z)<0\}$ if $Im(Z)<0$). Then, if $|Z|$ is greater than a constant large enough, we see that $g(B)\subset B$ ($g(B^{\pm})\subset B^{\pm}$ in the relevant cases), hence there is a fixed point of $g$ in $\overline{B}$ and therefore an element $z$ satisfying $z+\ln(z)=Z$. 

Furthermore, the pullback by $f$ of the form $\omega$ on $V_R$ gives $\omega'=(1+1/z)dz$. Then, the change of coordinate $w=1/z$ gives us that $(U_R,\omega')$ has a pole of order two at infinity with residue -1. Hence it is also the case for $(V_R,\omega)$.

Let $k\geq 2$. The pullback of the form $(1+1/z)dz$ by the map $z\to z^{k-1}$ gives $((k-1)z^{k-2}+(k-1)/z)dz$, \emph{i.e.} we get at infinity a pole of order $k$ with residue $-(k-1)$. In terms of the  flat metric, a neighborhood of a pole of order $k$ and residue $-(k-1)$ is just the natural cyclic $(k-1)$--covering of $V_R$. Then, suitable rotation and rescaling gives the local model for a pole of order $k$ with a nonzero residue.


\subsection{Moduli space}

If $(X,\omega)$ and $(X',\omega')$ are such that there is a biholomorphism $f:X\to X'$ with $f^* \omega'=\omega$, then $f$ is an isometry for the metrics defined by $\omega$ and $\omega'$. Even more, for the local coordinates defined by $\omega,\omega'$, the map $f$ is in fact a translation. 

As in the case of Abelian differentials, we consider the moduli space of meromorphic differentials, where $(X,\omega)\sim (X',\omega')$ if there is a biholomorphism $f:X\to X'$ such that $f^* \omega'=\omega$.  A stratum corresponds to prescribed degree of zeroes and poles. We denote by $\mathcal{H}(n_1^{\alpha_1},\dots ,n_r^{\alpha_r})$ the \emph{stratum} that corresponds to meromorphic differentials with $\alpha_i$ singularities of degree $n_i$. Such stratum is nonempty if and only if $\sum_{i=1}^r \alpha_i n_i=2g-2$ for some integer $g\geq 0$ and if there is not just one simple pole. 

We define the topology on this space in the following way: a small neighborhood of $S$, with conical singularities $\Sigma$, is defined to be the equivalence classes of surfaces $S'$ for which there is a differentiable injective map $f:S\backslash V(\Sigma)\to S'$ such that $V(\Sigma)$ is a (small) neighborhood of $\Sigma$, $Df$ is close the identity in the translation charts, and the complement of the image of $f$ is a union of disks. One can easily check that this topology is Hausdorff.

\subsection{Connected components of the moduli space of meromorphic differentials}
The connected components of the moduli space of meromorphic differentials were classified by the author in \cite{B:merom}. Here we recall this classification, and state some technical facts that appear in the proof, and that are necessary for this paper. First, recall the well known fact that any stratum of genus zero meromorphic differentials is connected since it corresponds more or less to a moduli space of marked points on the sphere.

Let $\gamma$ be a simple closed curve parametrized by the arc length on a translation surface that avoids the singularities. Then $t\to \gamma'(t)$ defines a map from $\mathbb{S}^1$ to $\mathbb{S}^1$. We denote by $Ind(\gamma)$ the index of this map.

Assume that the surface has genus one. Let $(a,b)$ be a pair of closed curves  representing a symplectic basis of the homology of $S$, then we define the \emph{rotation number} of $S$ as 
$$rot(S)=\gcd(Ind(a),Ind(b),n_1,\dots n_r,p_1,\dots ,p_s)$$
where $n_1,\dots ,n_r$ are the order of zeroes of $S$ and $p_1,\dots ,p_s$ are the degree of poles of $S$. We can show that it does not depend on the choice of $(a,b)$ and hence is an invariant of connected components. We have the following result.

\begin{theorem}\label{th:cc:g1}
Let $\mathcal{H}(n_1,\dots ,n_r,p_1,\dots ,p_s)$, with $n_i>0$, $p_j<0$ and   $\sum_{j}p_j<-1$ be a stratum of genus one meromorphic differentials. Let $d$ be a positive divisor of $N=\gcd(n_1,\dots ,n_r,p_1,\dots ,p_s)$. There is a unique connected component of $\mathcal{H}(n_1,\dots ,n_r,p_1,\dots ,p_s)$ with rotation number $d$, except when $r=s=1$ and $d=N$, in which case such a component does not exists.
\end{theorem}

A translation surface $S=(X,\omega)$ is \emph{hyperelliptic} if the underlying Riemann surface is hyperelliptic, \emph{i.e.} there is an involution $i$ such that $X/i$ is the Riemann sphere, and if $\omega$ satisfies $i^*\omega=-\omega$. 

Assume that the translation surface $S$ has only even degree singularities $S\in \mathcal{H}(2n_1,\dots ,2n_r,2p_1,\dots ,2p_s)$. Let $(a_i,b_i)_{i\in \{1,\dots ,g\}}$ be a collection of simple closed curves representing a symplectic basis of the homology of $S$. We define the \emph{spin structure} of $S$ as 
$$\sum_{i=1}^g (ind(a_i)+1)(ind(b_i)+1) \mod 2.$$

It is an invariant of connected components of the moduli space of meromorphic differentials. When the surface $S$ has only a pair of poles that are simple, and with even degree zeroes, \emph{i.e.} $S$ is in the stratum $\mathcal{H}(2n_1,\dots ,2n_r,-1,-1)$, it is also possible to define a ``spin structure'' invariant by considering a surface in $\mathcal{H}(2n_1,\dots ,2n_r)$ obtained after cutting the ends of the two infinite cylinders, and gluing them together (see \cite{B:merom}).

Note that an elementary computation shows that, when a surface of genus one has only even degree singularities, then it has an even spin structure if and only if its rotation number is odd.

In the next theorem, we say that the set of poles and zeroes is:
\begin{itemize}
\item of \emph{hyperelliptic type} if the degree of the zeroes are of the kind $\{2n\}$ or $\{n,n\}$, for some positive integer $n$, and if the degree of the poles are of the kind $\{2p\}$ or $\{p,p\}$, for some negative integer $p$.
\item of \emph{even type} if the degrees of zeroes are all even, and if the degrees of the poles are either all even, or are $\{-1,-1\}$.
\end{itemize}

\begin{theorem}\label{th:cc:g2}
Let $\mathcal{H}=\mathcal{H}(n_1,\dots ,n_r,p_1,\dots ,p_s)$, with $n_i>0$, $p_j<0$ be a stratum of genus $g\geq 2$ meromorphic differentials. We have the following.
\begin{enumerate}
\item If $\sum_{i} p_i$ is odd and smaller than -2, then $\mathcal{H}$ is nonempty and connected.
\item If $\sum_i p_i=-2$ and $g=2$, then:
\begin{itemize}
\item if the set of poles and zeroes is of hyperelliptic type, then there are two connected components, one hyperelliptic, the other not (in this case, these two components are also distinguished by the parity of the spin structure).
\item otherwise, the stratum is connected.
\end{itemize}
\item If $\sum_i p_i<-2$ or if $g>2$, then:
\begin{itemize}
\item  if the set of poles and zeroes is of hyperelliptic type, there is exactly one hyperelliptic connected component, and one or two nonhyperelliptic components that are described below. Otherwise, there is no hyperelliptic component.
\item if the set of poles and zeroes is of even type, then $\mathcal{H}$ contains exactly two nonhyperelliptic connected components that are distinguished by the parity of the spin structure. Otherwise $\mathcal{H}$ contains exactly one nonhyperelliptic component.
\end{itemize}
\end{enumerate}
\end{theorem}

The proof of these theorems involve some constructions, introduced first by Kontsevich and Zorich in \cite{KoZo}. These constructions are called \emph{breaking up a zero} and \emph{bubbling a handle}. We do not give a precise definition here since we will generalize them in Section~\ref{connected:sum}, but we summarize the important properties.
\begin{itemize}
\item \emph{Breaking up a zero} is a local surgery in a neighborhood of a singularity of order $n \geq 0$ (the metric is unchanged outside that neighborhood), that replaces that singularity by a pair of singularities of order $n_1,n_2 \geq 0$, with $n_1+n_2=n$. We can show (see \cite{B:merom}) that each connected component of the moduli space of meromorphic differentials can be obtained from a connected component of a stratum of the form $\mathcal{H}(n,p_1,\dots ,p_s)$ (with $n\geq 0$ and $p_1,\dots ,p_s<0$) after successive use of that surgery. In the case that either $n_1$ or $n_2$ is zero, we just add a marked point and the metric is unchanged.
\item \emph{Bubbling a handle} is a local  surgery in a neighborhood of a singularity of order $n\geq 0$, that replaces that singularity by a singularity of order $n+2$. The genus of the surface increases by one. We can show (see \cite{B:merom}) that each minimal connected component can be obtained starting from a genus zero stratum, by using this surgery repeatedly.
\end{itemize}

\subsection{Poles with zero residues}
The geometric constructions involved in Section~\ref{elem:moves} often require the use of a pole with zero residue. Here we give a necessary and sufficient condition for a connected component of stratum to contain a surface with a pole of a given order with zero residue.

The following lemma lists some well known cases where all poles necessarily have non-zero residues.
\begin{lemma}
Let $\omega$ be a meromorphic one-form on a closed Riemann surface $S$ and $P$ be a (non-simple) pole. Then, $P$ has necessarily nonzero residue in the following two cases.
\begin{itemize}
\item $S=\mathbb{CP}^1$ and $\omega$ has exactly two poles and a zero.
\item There exists exactly one other pole, which is simple.
\end{itemize}

\end{lemma}\label{nonzero:residue}
 \begin{proof}
For the first case:  let $p$ and $q$ be the degree of the poles. We identify $\mathbb{CP}^1$  with $\mathbb{C}\cup \{\infty \}$, and can assume that $P=0$,  the other pole is $1$, and the zero of $\omega$ is at $\infty $. Then, up to a multiple constant, 
 $\omega={z^p}{(1-z)^q}dz$, and we easily check that the residue at 0 is nonzero.
 
 For the second case, the residue of a simple pole is nonzero and if $P$ is the only other pole, it has opposite residue since by Stokes theorem the sum of residues of poles is zero.
 \end{proof}

 \begin{proposition}\label{zero:residue}
Let $\mathcal{C}\subset \mathcal{H}(n_1,\dots ,n_r,p_1,\dots ,p_s)$, with $n_i>0$, $p_j<0$ be a connected component of the moduli space of meromorphic differentials. We assume that there exists $p\in \{p_1,\dots ,p_s\}$, such that $p<-1$.
We assume that we are not in the case of the previous lemma.
Then,  there exists in $\mathcal{C}$ a flat surface with a pole of degree $p$ with zero residue.
 \end{proposition}
 \begin{proof}
The case is trivial when there is only one pole. 
In this proof, we will  assume first that there are exactly two poles of degree $p$ and $q$, (by assumption, we must have $p,q<-1$). This leads to the study of three cases, depending on the genus. Then, we will deal with the case of  at least three poles.

\emph{Case 1: two poles, genus zero}:
Since we are not in the case of the previous lemma, there are necessarily at least two zeroes. We start from $(\mathbb{CP}^1,z^{-p-2}dz)$, $(\mathbb{CP}^1,z^{-q-2}dz)$, then break the zero $P$ of the first one (resp. $Q$ of the second one) into a pair of zeroes $P_1,P_2$ of order $p_1,p_2 \geq 0$ with $p_1+p_2=-p-2$ (resp. $Q_1,Q_2$ of order $ q_1,q_2 \geq 0$ with $q_1+q_2=-q-2$), so that there is a vertical saddle connection $\gamma_1$ (resp. $\gamma_2$) of length $\varepsilon$ joining the two singularities. We obtain two surfaces $S_1$ and $S_2$. Then, cut $\gamma_1,\gamma_2$, and paste the left part of $\gamma_1$ (resp. $\gamma_2$) to the right part of $\gamma_2$ (resp. $\gamma_1$). This defines the segment $a$ (resp. $b$) in Figure~\ref{fig:2poles:noresidue:g0}. We obtain a flat surface in $\mathcal{H}(p,q,p_1+q_1+1,p_2+q_2+1)$. Choosing suitably $p_i,q_i$, we can obtain any stratum with two zeroes. Examples in the other strata are obtained from these examples by breaking up zeros.
Since each stratum in genus zero is connected, the case is proven.

\begin{figure}[htb]
\begin{tikzpicture}[scale=1]
\coordinate (vb) at (0,2);
\coordinate (mvb) at (0,-2);

\draw[dotted,fill=gray!10] (2,1)[] circle (2);
\draw[dotted,fill=gray!10] (10,1)[] circle (2);
\draw (2,0)--++(vb) node[midway,left] {$a$} node {\tiny $\bullet$} node[above] {$P_2$}--++(mvb) node[midway,right] {$b$} node {\tiny $\times$} node[below] {$P_1$};
\draw (10,0)--++(vb) node[midway,left] {$b$} node {\tiny $\bullet$} node[above] {$Q_2$} --++ (mvb) node[midway,right] {$a$}  node {\tiny $\times$} node[below] {$Q_1$}; 
\draw (0.5,1.5) node {$S_1$};
\draw (8.5,1.5) node {$S_2$};

\end{tikzpicture}
\caption{Surface of genus zero with two poles and no residue.}
\label{fig:2poles:noresidue:g0}

\end{figure}
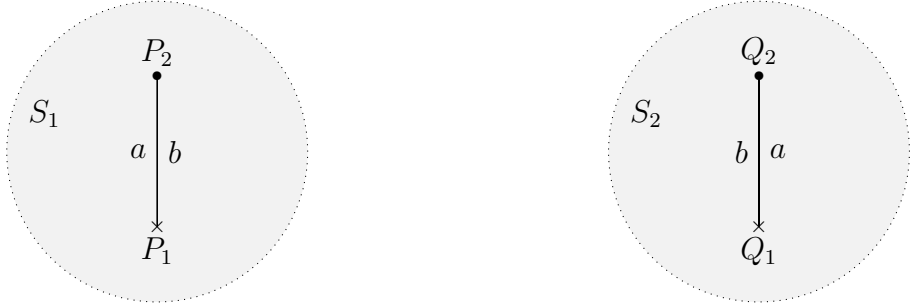

\medskip	

\emph{Case 2: two poles, genus one}:
We first build a suitable surface in any  component of the stratum $\mathcal{H}(p,q,-p-q)$ of genus one surfaces. We start from $S_0=(\mathbb{CP}^1,z^{-p-2}dz)$ and $S_2$ as previously. The surface  $S_0$ has a zero $P$ of order $-p-2$, and the surface $S_2$ has a pair of zeroes $Q_1,Q_2$ of orders $q_1,q_2$ with $q_1+q_2=-q-2$.

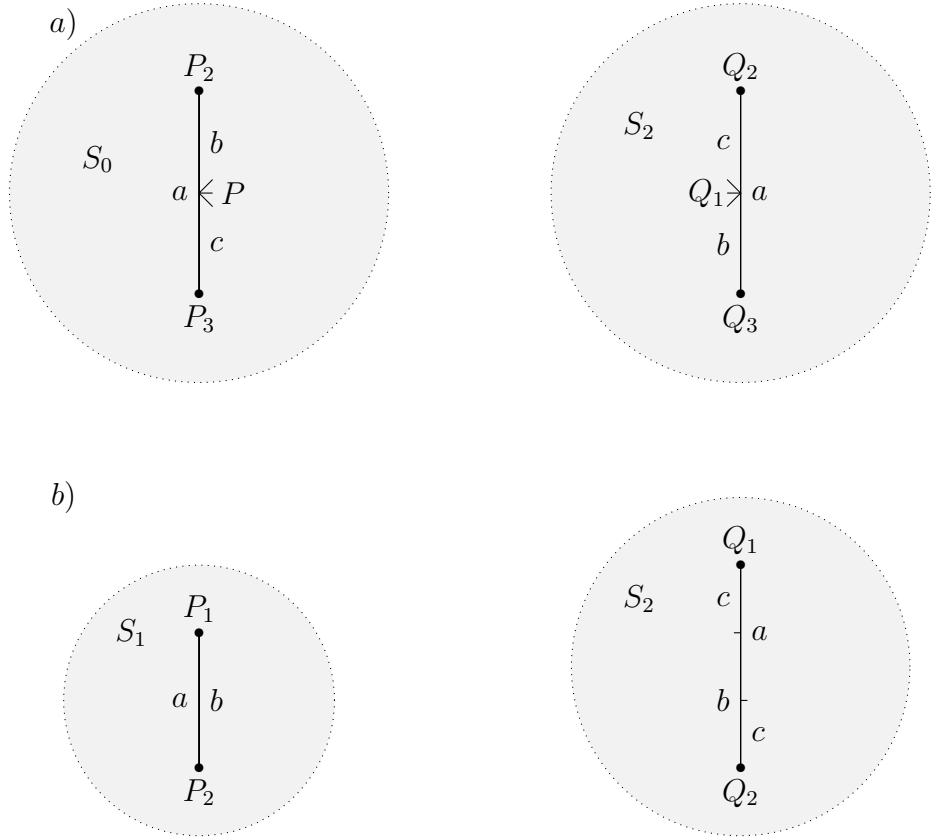
\begin{figure}[htb]
\begin{tikzpicture}[scale=0.9]
\coordinate (v) at (0,2);
\coordinate (mv) at (0,-2);
\coordinate (v1) at (0,1);
\coordinate (mv1) at (0,-1);
\coordinate (vb) at (0,3);
\coordinate (mvb) at (0,-3);
\coordinate (v1b) at (0,1.5);
\coordinate (mv1b) at (0,-1.5);

\draw[dotted,fill=gray!10] (2,1.5)[] circle (2.8);
\draw[dotted,fill=gray!10] (10,1.5)[] circle (2.8);
\draw (2,0)--++(vb) node[midway,left] {$a$} node {\tiny $\bullet$} node[above] {$P_2$}--++(mv1b) node[midway,right] {$b$} --++ (mv1b) node[midway,right] {$c$}node {\tiny $\bullet$} node[below] {$P_3$};
\draw (2.5,1.5) node {$P$};
\draw (10,0)--++(v1b) node[midway,left] {$b$} --++(v1b) node[midway,left] {$c$} node {\tiny $\bullet$} node[above] {$Q_2$} --++ (mvb) node[midway,right] {$a$}  node {\tiny $\bullet$} node[below] {$Q_3$}; 
\draw (9.5,1.5) node {$Q_1$};
\draw (9.8,1.7)--(10,1.5)--(9.8,1.5);
\draw (9.8,1.3)--(10,1.5);
\draw (2.2,1.7)--(2,1.5)--(2.2,1.5);
\draw (2.2,1.3)--(2,1.5);
\draw (0.5,2) node {$S_0$};
\draw (8.5,2.5) node {$S_2$};
\draw (0,4) node {$a)$};

\draw[dotted,fill=gray!10] (2,-6)[] circle (2);
\draw[dotted,fill=gray!10] (10,-5.5)[] circle (2.5);
\draw (2,-7)--++(v) node[midway,left] {$a$} node {\tiny $\bullet$} node[above] {$P_1$}--++(mv) node[midway,right] {$b$} node {\tiny $\bullet$} node[below] {$P_2$};
\draw (10,-7)--++(v) node[midway,left] {$b$} --++(v1) node[midway,left] {$c$} node {\tiny $\bullet$} node[above] {$Q_1$} --++ (mv) node[midway,right] {$a$}--++ (mv1) node[midway,right] {$c$} node {\tiny $\bullet$} node[below] {$Q_2$}; 
\draw (10,-5)--(9.9,-5);
\draw (10,-6)--(10.1,-6);
\draw (1,-5) node {$S_1$};
\draw (8.5,-4.5) node {$S_2$};
\draw (0,-3) node {$b)$};                   
\end{tikzpicture}
\caption{Surfaces in $\mathcal{H}(-p,-q,p+q)$.}
\label{fig:2poles:noresidue}
\end{figure}

Consider a metric segment $[P_2,P_3]$ on $S_0$, with  $P$ on its middle, and such that one of the angular sectors at $P$ defined by this segment has angle $\pi$ (see  Case $a)$ of Figure~\ref{fig:2poles:noresidue}). Similarly, we consider a segment $[Q_2,Q_3]$ on $S_2$, with $Q_1$ on its middle and the same condition on the angular sector at $Q_1$.
We remark that such segment exists, since $Q_1, Q_2$ are obtained after breaking up a singularity, and in this case, there is by construction (see \cite{KoZo}) a segment joining $Q_1$ to $Q_2$ that we can assume to be arbitrarily small. We can assume that the two segments are vertical, isometric, and with opposite orientation. Then, cutting the surfaces along these segments, and gluing them according to Figure~\ref{fig:2poles:noresidue}, one gets a surface $S$ in  $\mathcal{H}(p,q,-p-q)$ (note that the midpoint of $a$ is not  a singularity). 
 We must check that all connected components of this stratum are obtained. We first consider a basis for the homology of $S$: consider the saddle connection $\gamma_b$ corresponding to $b$, joining the unique conical singularity to itself. At the singularity, it defines a sector of angle $2\pi(1+q_2+1+(-p-2))+\pi=(-2p+2q_2+1)\pi$. We take a smooth path $\eta_b$ homotopic to $\gamma_b$ that avoids the singularity. It has index $-p+q_2$ (or $-p-q-(-p+q_2)$). Similarly, we have a smooth path $\eta_c$, homotopic to the  saddle connection corresponding to the segment $c$ with index $-p$, and $\eta_b,\eta_c$ define a symplectic basis of $S$. So the rotation number of the surface is $\gcd(q_2, -p,-q)$, with $q_2$ that can be any integer in  $\{0,\dots , -q-2\}$. If $-q>-p$, we can clearly obtain any divisor of $\gcd(p,q,-p-q)$, so we obtain any connected component. When  $p=q$, one cannot obtain in this way the component with rotation number $-p-1$. But since the rotation number must divide $p$, we are in the case $p=q=-2$. 
 In this case, we glue two Euclidean planes as in Figure~\ref{fig:2poles:noresidue}, $b)$. Here, paths $\eta_a$ and $\eta_c$ define a symplectic basis of $S$, and we see that the rotation number is $1$, since the index of $\eta_a$ is 1. Finally, once obtained any connected component of $\mathcal{H}(p,q,-p-q)$, breaking up the zero in a suitable way gives any component of any stratum of genus one with two non-simple poles.

\emph{Case 3: two poles, higher genus}:
Here, suitably bubbling handles from genus one surfaces leads to any minimal connected component in higher genus, and breaking up the zero leads to any connected component of the moduli space of meromorphic differentials.

\emph{Case 4: at least three poles}:
\begin{figure}[htb]
\begin{tikzpicture}[scale=1]
\coordinate (v) at (0,1);
\coordinate (mv) at (0,-1);

\draw[dotted,fill=gray!10] (1,1.5)[] circle (1);
\draw[dotted,fill=gray!10] (4,1.5)[] circle (1);
\draw[dotted,fill=gray!10] (8,1.5)[] circle (1);
\draw[dotted,fill=gray!10] (11,1.5)[] circle (1);
\draw (4,1)--++(v) node[midway,left] {$a_0$} node {\tiny $\bullet$} --++(mv) node[midway,right] {$a_1$} node {\tiny $\bullet$} ;
\draw (8,1)--++(v) node[midway,left] {$a_{k-1}$} node {\tiny $\bullet$} --++(mv) node[midway,right] {$a_{k}$} node {\tiny $\bullet$} ;
\fill [fill=gray!00] (-0.5,1)--(1,1) -- (1,2) -- (-0.5,2)--cycle;
\draw (0.1,1)--(1,1)  node[midway,below] {$l_0$} node  {\tiny $\bullet$}  -- (1,2) node[midway,right] {$a_0$} node  {\tiny $\bullet$} -- (0.1,2) node[midway,above] {$l_0$};
\fill [fill=gray!00] (13,1)--(11,1) -- (11,2) -- (13,2)--cycle;
\draw (11.9,1)--(11,1)  node[midway,below] {$l_{k+1}$} node  {\tiny $\bullet$} -- (11,2) node[midway,left] {$a_{k}$} node  {\tiny $\bullet$} -- (11.9,2) node[midway,above] {$l_{k+1}$};
\draw (1,0) node {$S_0$};                   
\draw (4,0) node {$S_1$};
\draw (8,0) node {$S_k$};                   
\draw (11,0) node {$S_{k+1}$};   
\draw (6,1.5) node {\dots }  ;              

\end{tikzpicture}
\caption{Surface of genus zero in $\mathcal{H}(n,p_0,\dots ,p_{k+1})$ with $k$ poles of zero residues and two poles of nonzero residues.}
\label{fig:2poles:noresidue:3p}
\end{figure}
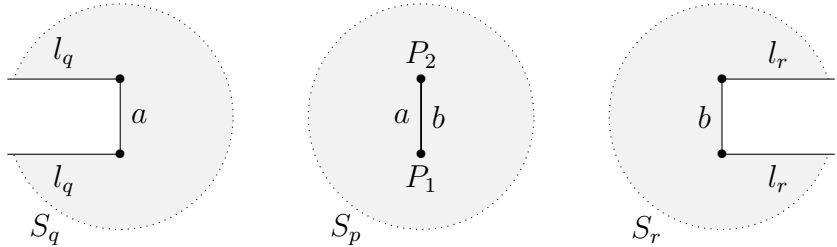

Let $k\geq 1$, we first build a genus zero surface with $2+k$ poles of degree $p_0,\dots ,p_{k+1}$ respectively. We assume that $p_0,\dots ,p_{k+1}<-1$.
We start from spheres $S_{0},\dots ,S_{k+1}$ with exactly one pole of degree $p_0,\dots ,p_{k+1}$ respectively and one zero (of degree $-p_0-2,\dots ,-p_{k+1}-2$ respectively). 
Consider on $S_{0}$ an infinite horizontal segment $l_0$ joining the zero to the pole $P_0$ ($l_0$ is chosen so that it identifies by a translation map to the half-line $]-\infty ,0[$), then consider the half-infinite horizontal band of width $1$ with bottom side $l_0$. Cut this band, and glue together by translation the two horizontal sides. One obtains a surface, still denoted $S_0$ with a small vertical boundary component of length $1$, and the pole $P_0$ has now a  nonzero residue. We do the same for $S_{k+1}$, but starting from an half-line $l_{k+1}$ that identifies to $]0,\infty [$. Then, on each $S_{i}$, we cut along a vertical segment of length $1$, that is attached to a singularity. Then, as in Figure~\ref{fig:2poles:noresidue:3p}, we glue by translation a vertical boundary segment of $S_i$ to the corresponding one of $S_{i-1}$, and the other vertical boundary segment of $S_i$ to the corresponding one of $S_{i+1}$. This defines a (closed) flat sphere with poles $P_1,\dots ,P_k$ of zero residue, two poles $P_0,P_{k+1}$ of nonzero residue, and a single singularity of positive degree. 
The surfaces $S_0, S_{k+1}$ can be replaced without difficulties by half-infinite cylinders, hence we can have $P_0$ or $P_{k+1}$ that are simple poles.

If we want to have more simple poles, we start from a horizontal half-line  joining $P_0$ to the singularity of positive degree, then cut along this line, consider a half-infinite horizontal band, and glue together each infinite horizontal side of the band to the corresponding half-line on the surface. We obtain a translation surface with a boundary component which is a vertical segment. We glue on this segment a half-infinite cylinder so that we obtain a translation surface with no boundary component. This procedure adds a simple pole, modifies the residue of $P_0$, and leaves invariant the other residues. Repeating this procedure, we obtain as many simple poles as we want, and we only change the residue of $P_0$.

Now, suitably bubbling handles and breaking up zeroes, we obtain any connected component with at least three poles, and this does not change the residues of $P_1,\dots ,P_k$.
\end{proof}

\section{Moduli spaces of framed meromorphic differentials}\label{mod:space:framed}
As in \cite{B:labeled},  a \emph{frame} on a translation surface $S$ is a map $F_S$ from a finite alphabet $\mathcal{A}$ to a discrete combinatorial data of $S$.

For a suitable collection of frames on translation surfaces in a stratum $\mathcal{H}(n_1^{\alpha_1},\dots ,n_r^{\alpha_r})$, we define the corresponding moduli space of framed surfaces by identifying $(S,F_S)$ and $(S', F_{S'})$ if there is a translation mapping from $S$ to $S'$ that is consistent with the frames.

We are interested in two cases.
 The first case is when the alphabet $\mathcal{A}$ admits a partition $\sqcup_{i=1}^r \mathcal{A}_i$ with $|\mathcal{A}_i|=\alpha_i$ and the collection of frames we consider are all possible one-to-one maps $F_S$ from $\mathcal{A}$ to the set of singularities of $S$ such that, for all $ i$, for all $ P\in \mathcal{A}_i$, $F_S(a)$ is a singularity of $S$ of degree $n_i$. We obtain the moduli space of translation surface with named singularities $\mathcal{H}^{sing}(n_1^{\alpha_1},\dots ,n_r^{\alpha_r})$.

The following proposition will be useful.
\begin{proposition}\label{cc:marked:sing}
The connected components of $\mathcal{H}^{sing}(n_1^{\alpha_1},\dots ,n_r^{\alpha_r})$ are in one-to-one correspondence with the connected components of $\mathcal{H}(n_1^{\alpha_1},\dots ,n_r^{\alpha_r})$.
\end{proposition}
\begin{proof}
This is clearly the case for genus zero stratum. Otherwise, we use the fact that each connected component of $\mathcal{H}^{sing}(n_1^{\alpha_1},\dots ,n_r^{\alpha_r})$ is adjacent to the minimal stratum obtained by coalescing all singularities of positive degree. Then the proof is similar to the proof of Proposition~7.2 of \cite{B:merom}. See also the connected sum construction of this paper in Section~\ref{connected:sum}.
\end{proof}

Now we define a more specific combinatorial datum for a singularity.
\begin{definition}
Let $P$ be a singularity of $S$ which is not a simple pole. A (right) \emph{horizontal separatrix} for $P$ is an equivalence class of horizontal geodesics $\gamma:]a,b[\to S$, satisfying $\gamma'=1$, $\lim_{t\to a} \gamma(t)=P$ with the following conditions:
\begin{itemize}
\item if $\deg(P)>0$: $a=0$ and $\gamma_1\sim\gamma_2$ if they coincide on a subinterval of the form $]0,\varepsilon[$.
\item if $\deg(P)<-1$:  $a=-\infty$, and $\gamma_1\sim\gamma_2$ if the distance for the euclidean metric between $\gamma_1(t)$ and $\gamma_2(t)$ is bounded as $t$ tends to $-\infty $. 
\end{itemize}
\end{definition}

\begin{remark}
For a singularity $P$ (pole or zero), the number of possible choices of horizontal separatrices is $|\deg(P)+1|$.
\end{remark}

Recall that two translation surfaces $S_1$ and $S_2$ define the same element in the moduli space if and only if  there is a one-to-one map $f:S_1\to S_2$ which is a translation. In particular such map sends a horizontal separatrix attached to $P\in S_1$ to a horizontal separatrix attached to $f(P)\in S_2$.

Now we define the second moduli space of framed meromorphic differentials. It corresponds to choosing a horizontal separatrix for each singularity. More precisely, 
we still assume that $\mathcal{A}$ admits a partition $\sqcup_{i=1}^r \mathcal{A}_i$ with $|\mathcal{A}_i|=\alpha_i$ and the collection of frames we consider are all possible maps $\tilde{F}_S$ such that:
\begin{itemize}
\item if $n_i\neq -1$ then for all $ P\in \mathcal{A}_i$, $\tilde{F}_S(P)$ is a horizontal separatrix for a singularity of degree $n_i$. 
\item if $n_i=-1$ then for all $ P\in \mathcal{A}_i$, $\tilde{F}_S(P)$ is a singularity of degree $-1$.
\item if $P\neq Q$, then the singularity corresponding to  $\tilde{F}_S(P)$ is different from the singularity corresponding to  $\tilde{F}_S(Q)$.
\end{itemize}
In particular, two framed surfaces $(S_1,\tilde{F}_1)$ and $(S_2,\tilde{F}_2)$ represent the same element in the moduli space if:

\begin{itemize}
\item there is an one-to-one map $f:S_1\to S_2$ that is a translation.
\item $f\circ \tilde{F}_{S_1}=\tilde{F}_{S_2}$, \emph{i.e.} for each $P\in \mathcal{A}$,
\begin{itemize}
\item the image of the singularity labeled $P$ of $S_1$ is the singularity labeled $P$ of $S_2$,
\item the image by $f$ of the marked horizontal separatrix of the singularity labeled $P$ of $S_1$ is corresponding  horizontal separatrix labelled $P$ of $S_2$ (when the singularity is not a simple pole). 
\end{itemize}
\end{itemize}

We obtain the moduli space of translation surfaces with marked horizontal separatrices $\mathcal{H}^{hor}(n_1^{\alpha_1},\dots ,n_r^{\alpha_r})$.

Denote by $\pi_h: \mathcal{H}^{hor}(n_1^{\alpha_1},\dots ,n_r^{\alpha_r})\to \mathcal{H}^{sing}(n_1^{\alpha_1},\dots ,n_r^{\alpha_r})$ and $\pi_s: \mathcal{H}^{sing}(n_1^{\alpha_1},\dots ,n_r^{\alpha_r})\to \mathcal{H}(n_1^{\alpha_1},\dots ,n_r^{\alpha_r})$ the coverings obtained by forgetting the horizontal separatrices, and the names of the singularities, respectively. 

Let $\mathcal{C}$ be a connected component of $\mathcal{H}(n_1^{\alpha_1},\dots ,n_r^{\alpha_r})$. From Proposition~\ref{cc:marked:sing},  $\mathcal{C}^{sing}=\pi_s^{-1}(\mathcal{C})$ is connected.  We define $\mathcal{H}_\mathcal{C}^{hor}$ the subset of $\mathcal{H}^{hor}(n_1^{\alpha_1},\dots ,n_r^{\alpha_r})$ whose underlying translation surfaces are in $\mathcal{C}$. We  have $\mathcal{H}_\mathcal{C}^{hor}=(\pi_h)^{-1}\mathcal{C}^{sing}$.
We want to compute the number of connected components of $\mathcal{H}_\mathcal{C}^{hor}$.

We choose ${S}_b^{hor}$ a base element of $\mathcal{H}_\mathcal{C}^{hor}$ and $S_b^{sing}=\pi_h(S_b^{hor})$ the corresponding flat surface (with marked singularities). By definition, for each element $S$ of $\pi_h^{-1}(S_b^{sing})$, and each element $P\in \mathcal{A}$, $\tilde{F}_S(P)$ is a horizontal separatrix attached to the same singularity as $\tilde{F}_{S_b^{hor}}(P)$.

For each singularity $P$ of $S_b^{sing}$ which is not a simple pole, the set of horizontal separatrices, once fixed the corresponding one of $S_b^{hor}$, is  identified to the cyclic group $\mathbb{Z}/ (\deg(P)+1)\mathbb{Z}$ in the following way:
\begin{enumerate}
\item if $P$ is a conical singularity, we go from a separatrix to the next one by considering a small counterclockwise arc of angle $2\pi$ around the singularity. Hence, the identification is just given by the cyclic (counterclockwise) order around $P$. 
\item if $P$ is an nonsimple pole, we go from a separatrix to the next one by considering a large counterclockwise arc of angle $2\pi$. In particular, from the point of view of the pole $P$, it corresponds to the cyclic \emph{clockwise} order around $P$.
\end{enumerate}


We consider the action of the monodromy group of the covering $\pi_h$, on the fiber $\pi_h^{-1}(S_b^{sing})$, \emph{i.e.} all the possible frames on the surface $S_b^{sing}$.

\begin{definition}
Let $Hor$ be the group $\prod_{P}\mathbb{Z}/(\deg(P)+1)\mathbb{Z},$
where the product is taken over all singularities of degree different from~$-1$. From above, $Hor$ is  identified with the fiber $\pi_h^{-1}(S_b^{sing})$. We define $Mon$ to be  the image of the element $(0,\dots ,0)\in Hor$ for the monodromy action.
\end{definition}

\begin{lemma}
The action of the monodromy group of the covering $\pi_h$ on $Hor$ corresponds to the addition of the corresponding elements of $Mon$.  In particular, $Mon$ is a subgroup of $Hor$.
\end{lemma}
\begin{proof}
 Let $(l_P)_{P} \in Hor$ and let $\gamma=(S_t)_{t\in [0,1]}$ be a closed path in $\mathcal{C}^{sing}$ with $S_0=S_b^{sing}$, and that defines a element $(x_P)_{P}\in Mon$. For each $P$, the angle between two separatrices of $P$ is preserved along $\gamma$, hence the image of the action of $\gamma$ on $(l_P)_P$ is $(l_P+x_P)_P$. 
\end{proof}

Since $\mathcal{C}^{sing}$ is connected, we can therefore identify connected components of $\mathcal{H}_\mathcal{C}^{hor}$ with cosets of $Mon$ in $Hor$. This proves the following:
\begin{corollary}
 The number of connected component of $\mathcal{H}_\mathcal{C}^{hor}$ is the index of the subgroup $Mon$ of $Hor$.
\end{corollary}

\begin{definition}
Let $Hor$ be the group defined above, and let $P$ be a singularity of the surface $S_b^{hor}$. We denote by $\delta_P$ the element in $Hor$ which is $1$ on the factor corresponding to $P$, and zero everywhere else. If $P$ is a simple pole, then $\delta_P=0$. For any singularity, we have $(\deg(P)+1)\delta_P=0$.
\end{definition}

The goal of the next  section is to prove the following two propositions, which give a collection of elements that are in $Mon$.

\begin{proposition}\label{prop:g0}
Let $S_b^{hor}$ be a framed genus zero translation surface.
 Let $P,Q$ be a pair of singularities of  $S_b^{hor}$. We have
 $$\tau_{P,Q}:=\deg(Q) \delta_P+\deg(P)\delta_Q\in Mon.$$
\end{proposition}
Note that $\tau_{P,Q}=(\deg(P)+\deg(Q)+1)(\delta_P+\delta_Q)$.

\begin{proposition}\label{prop:g}
Let $S_b^{hor}$ be a framed translation surface of genus $g\geq 1$, such that the underlying translation surface is neither in a hyperelliptic connected component, nor in a stratum of the kind $\mathcal{H}(-1,-1,n_1,\dots ,n_r)$, with $n_i>0$.
\begin{enumerate}
\item Let $P,Q$ be a pair of singularities of $S_b^{hor}$. We assume that neither $P$ nor $Q$ is the only pole of $S_b^{hor}$.
We have
 $$\tau_{P,Q}:=\deg(Q) \delta_P+\deg(P)\delta_Q\in Mon.$$
 \item Let $P$ be a singularity of   $S_b^{hor}$. We have 
 $$\sigma_{P}:=2\delta_P\in Mon.$$
\end{enumerate}
If the underlying translation surface is in a nonhyperelliptic connected component of a stratum of the kind $\mathcal{H}(-1,-1,n_1,\dots ,n_r)$, with $n_i>0$, then the previous statement is still true if we assume that neither $P$ nor $Q$ are poles.
\end{proposition}



\section{Some elementary moves}\label{elem:moves}
\subsection{Connected sums} \label{connected:sum}
Let $S, S'$ be translation surfaces. Let $N\in S$, be a singularity of degree $n\geq 0$ and let $N'\in S'$ be a singularity of degree $n'=-2-n<0$. We assume that $N'$ has zero residue.  A pointed neighborhood $V\backslash\{N'\}$ of $N'$ is isometric to the complement of a metric disk  centered in 0, in the cone defined by the form $z^{-2-n'}dz=z^{n}dz$. Hence, after scaling (shrinking) appropriately the surface $S'$ so that this metric disk is isometric (as a translation surface) to a neighborhood $U$ of $N$, we can glue together $S\backslash U$ and $S'\backslash V$ along their boundaries to get a translation surface. This surgery is a flat version of the topological connected sum of two surfaces.
If $n\leq -2$, $n'\geq 0$, the construction is the same by reversing the roles of $S,S'$.
If $n=-1$, then  $n'=-1$, we must assume that the two simple poles have opposite residues, we obtain two half infinite cylinders with isometric waist curves. Cutting and pasting along such waist curves gives the required surface.

We are interested in some particular cases for $S'$, where it generalizes the two surgeries introduced by Kontsevich and Zorich in \cite{KoZo}, \emph{breaking up a singularity} and \emph{bubbling a handle}.

If $S'$ is a sphere with three singularities, \emph{i.e.} $S'\in H(-n-2, n_1, n_2)$, the above construction, when possible, replaces the singularity of degree $n$ by a pair of singularities of degree $n_1,n_2$ with $n=n_1+n_2$. 
\begin{itemize}
\item If $n\geq 0$ and $n_1,n_2\geq 0$, the construction is always possible and is precisely ``breaking up a singularity'' (see \cite{KoZo}). 
\item If $n\leq -1$, the construction is always possible, and breaks up the pole of degree $n$ into a pair of singularities of degree $n_1$ and $n_2$.
\item If $n\geq 0$ and either $n_1$ or $n_2$ is negative. Say $n_1< 0$ and $n_2\geq 0$.  The above construction is not possible since, $S'$ is a sphere with two poles (of respective degree $-n-2$ and $n_1$) and a zero, and in the case, the pole of degree $-n-2$ would have zero residue which would contradict Lemma~\ref{nonzero:residue}.
\end{itemize}

If $S'$ is a torus in $\mathcal{H}(-n-2,n+2)$, then the surgery adds a handle to the surface $S$, and the singularity of degree $n$ is replaced by a singularity of degree $n+2$. 
\begin{itemize}
\item For $n\geq 0$, the construction is always possible, and if we choose $S'$ that contains a cylinder, this is precisely the surgery ``bubbling a handle'' (see \cite{KoZo}).
\item For $n=-3$ or $n=-1$, $\mathcal{H}(-n-2,n+2)=\mathcal{H}(-1,1)=\emptyset$, so the construction does not make sense.
\item If $n\leq -4$ or $n=-2$,  the construction works well as soon as the pole of degree $n$ has zero residue. 
\end{itemize}

\begin{remark}
When $n'<0$ and $N'$ has nonzero residue, the above construction is not possible since the boundary of a pointed neighborhood of $N'$ is never isometric to the boundary of a neighborhood of $N$. However, once a proper rescaling (shrinking) of the surface $S'$ is done, it is possible to perform a surgery on $S$ that creates  a geodesic boundary component (``hole'') adjacent to the singularity $N$, so that the boundary of a neighborhood of $N$ becomes isometric to the boundary of a neighborhood of $N'$, making the construction doable, 
see Section~\ref{nonlocal:move}. Note that if $S$ has no poles, then due to Stokes theorem, this necessarily creates on $S$ at least one other boundary component. This idea has been continued in \cite{BCG+}.
\end{remark}


\subsection{A realization of $\tau_{P,Q}$, local case.}\label{local:move}
Consider a translation surface $S^{hor}$ with labeled horizontal separatrices. Let $P$ and $Q$ be two singularities of degree $p$ and $q$ respectively. Assume that $P,Q$ are obtained after the surgery ``breaking up a singularity'' above, with either $p,q\geq 0$, or $p+q\leq -2$, in the zero residue case. By construction, the singularities $P,Q$ are on a metric disc whose boundary is a covering of Euclidean circle. Cutting the surface along the circle and rotating the disc by an angle $\theta$, one gets a family of surfaces $(S_\theta)$. For $\theta=2\pi(p+q+1)$, one has $S_\theta=S$. Keeping track of the marked horizontal separatrices, we see at the end that the marked horizontal separatrices for $P,Q$ have changed by an angle $2\pi(p+q+1)$, and the horizontal separatrices of the other singularities have not changed. 

Now assume that $S_b^{hor}$ is in the same connected component as a translation surface $S^{hor}$ as above. Then, conjugating the above transformation with a path joining $S_b^{hor}$ to $S^{hor}$ gives the element $(p+q+1)(\delta_P+\delta_Q)=q \delta_P+p\delta_Q=\tau_{P,Q}$
 in $Mon$ (recall that $(p+1)\delta_P=(q+1)\delta_Q=0$).

\subsection{A realization of $\tau_{P,Q}$, nonlocal case.}\label{nonlocal:move}

The above transformation fails if $p+q\geq 0$ and either $p$ or $q$ is negative. 

Here we describe a (nonlocal) surgery that produces the same effect on the set of separatrices. We must first describe a way to do a connected sum  with a surface in $\mathcal{H}(p,q,-2-p-q)$. The idea is to make a ``hole'' (\emph{i.e.} a geodesic boundary component) adjacent to the singularity of degree $p+q$. The transformation is then obtained by continuously rotating the hole by an angle $2\pi(p+q+1)$.

We start from a surface $S_0$ in the stratum obtained by coalescing $P$ and $Q$. We assume that this is not a stratum of holomorphic differentials. We can assume that $S_0$ does not have any vertical saddle connections. Then, it is obtained by the \emph{infinite zippered rectangle construction}. We refer to  \cite{B:merom}, Section~3.3 for a precise construction, and present an example here (see Figure~\ref{ex:infinite:zipp}). Note that in this figure, the parameters $z_1,\dots ,z_4$ have positive real part and may take any value satisfying this condition, and $l_1,\dots ,l_4$ are horizontal half-lines.

\begin{figure}[htb]
\begin{tikzpicture}[scale=0.4]
\coordinate (z1) at (1,0.5);
\coordinate (z2) at (1,-1.5);
\coordinate (z3) at (2,2.5) ;
\coordinate (z4) at (1.5,-1) ;
\fill [fill=gray!20] (0,0)--++(3.5,0) --++ (z1) --++ (z2) --++(3.5,0) --++(0,4)--++(-9,0)--cycle;

\draw (0,0)--++(3.5,0) node[midway,above] {$l_1$} node {\tiny $\bullet$}
                     --++ (z1) node[midway,above] {$z_1$} node {\tiny $\bullet$}
                     --++ (z2) node[midway,above right] {$z_2$} node {\tiny $\circ$}
                     --++(3.5,0) node[midway, above] {$l_2$};
                     
\draw [dashed] (0,0) --++(1,0);                     

\fill [fill=gray!20] (12,-1)--++(3.5,0) --++ (z3) --++ (z4) --++(3.5,0) --++(0,2.5)--++(-10.5,0)--cycle;

\draw (12,-1)--++(3.5,0) node[midway,above] {$l_3$} node {\tiny $\circ$}
                     --++ (z3) node[midway,above left] {$z_3$} node {\tiny $\circ$}
                     --++ (z4) node[midway,above] {$z_4$} node {\tiny $\bullet$}
                     --++(3.5,0) node[midway, above] {$l_4$};
                     
\fill [fill=gray!20] (0,-3)--++(3.5,0) --++ (z2) --++ (z3) --++(2.5,0) --++(0,-4)--++(-9,0)--cycle;

\draw (0,-3)--++(3.5,0)node[midway,below] {$l_1$} node {\tiny $\bullet$}
                     --++ (z2) node[midway,below left] {$z_2$} node {\tiny $\circ$}
                     --++ (z3) node[midway,below right] {$z_3$} node {\tiny $\circ$}
                     --++(2.5,0) node[midway, below] {$l_2$};
                     
\fill [fill=gray!20] (12,-2)--++(3.5,0) --++ (z4) --++ (z1) --++(4.5,0) --++(0,-3.5)--++(-10.5,0)--cycle;

\draw (12,-2)--++(3.5,0)node[midway,below] {$l_3$} node {\tiny $\circ$}
                     --++ (z4) node[midway,below left] {$z_4$} node {\tiny $\bullet$}
                     --++ (z1) node[midway,below right] {$z_1$} node {\tiny $\bullet$}
                     --++(4.5,0) node[midway, below] {$l_4$};

\draw [dashed] (4.5,0.5) --++ (0,2) node[midway, above left] {$l$};          
                     
\end{tikzpicture}
\caption{Infinite zippered rectangle representation of a surface in  $\mathcal{H}(-2,-2,2,2)$.}
\label{ex:infinite:zipp}
\end{figure}
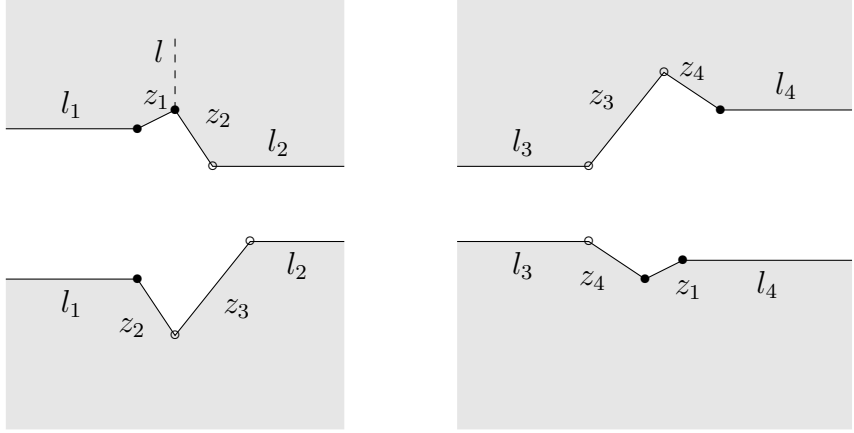

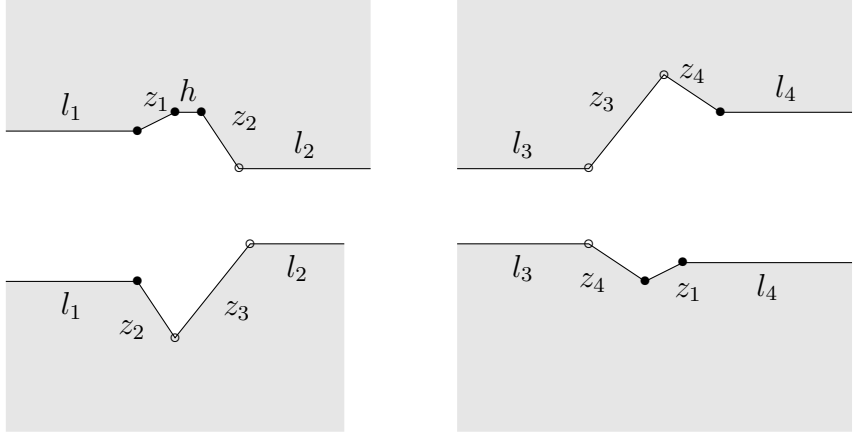
\begin{figure}[htb]
\begin{tikzpicture}[scale=0.4]
\coordinate (z1) at (1,0.5);
\coordinate (z2) at (1,-1.5);
\coordinate (z3) at (2,2.5) ;
\coordinate (z4) at (1.5,-1) ;
\coordinate (h) at (0.7,0);
\fill [fill=gray!20] (0,0)--++(3.5,0) --++(z1) --++ (h) --++ (z2) --++(3.5,0) --++(0,4)--++(-9.7,0)--cycle;

\draw (0,0)--++(3.5,0) node[midway,above] {$l_1$} node {\tiny $\bullet$}		 
                     --++ (z1) node[midway,above] {$z_1$} node {\tiny $\bullet$}
                       --++ (h) node[midway,above] {$h$} node {\tiny $\bullet$}
                     --++ (z2) node[midway,above right] {$z_2$} node {\tiny $\circ$}
                     --++(3.5,0) node[midway, above] {$l_2$};                    

\fill [fill=gray!20] (12,-1)--++(3.5,0) --++ (z3) --++ (z4) --++(3.5,0) --++(0,2.5)--++(-10.5,0)--cycle;

\draw (12,-1)--++(3.5,0) node[midway,above] {$l_3$} node {\tiny $\circ$}
                     --++ (z3) node[midway,above left] {$z_3$} node {\tiny $\circ$}
                     --++ (z4) node[midway,above] {$z_4$} node {\tiny $\bullet$}
                     --++(3.5,0) node[midway, above] {$l_4$};
                     
\fill [fill=gray!20] (0,-3)--++(3.5,0) --++ (z2) --++ (z3) --++(3.2,0) --++(0,-4)--++(-9.7,0)--cycle;

\draw (0,-3)--++(3.5,0)node[midway,below] {$l_1$} node {\tiny $\bullet$}
                     --++ (z2) node[midway,below left] {$z_2$} node {\tiny $\circ$}
                     --++ (z3) node[midway,below right] {$z_3$} node {\tiny $\circ$}
                     --++(3.2,0) node[midway, below] {$l_2$};
                     
\fill [fill=gray!20] (12,-2)--++(3.5,0) --++ (z4) --++ (z1) --++(4.5,0) --++(0,-3.5)--++(-10.5,0)--cycle;

\draw (12,-2)--++(3.5,0)node[midway,below] {$l_3$} node {\tiny $\circ$}
                     --++ (z4) node[midway,below left] {$z_4$} node {\tiny $\bullet$}
                     --++ (z1) node[midway,below right] {$z_1$} node {\tiny $\bullet$}
                     --++(4.5,0) node[midway, below] {$l_4$};

                     
\end{tikzpicture}
\caption{Same surface as before, after creating a hole.}
\label{hole:1}
\end{figure}


We choose a vertical separatrix $l$ adjacent to a singularity of degree $p+q$  (the dashed line in the above picture). Then, insert an horizontal saddle connection as in Figure~\ref{hole:1}. This creates a hole on the surface orthogonal to the direction $l$.  This resulting surface, for a suitable rescaling and a parameter $h$ small enough, can be glued as in Section~\ref{connected:sum} to a flat sphere  $S_1$ in $\mathcal{H}(p,q,-2-p-q)$, where the pole of degree $-2-p-q$ has suitable residue.

Now, we can continuously rotate the segment $h$ by an angle $\pi$ and obtain the surface with a hole that would be obtained from the separatrix $l'$ obtained after rotating $l$ by~$\pi$. 
This operation is described in Figure~\ref{hole:rotating}.

\begin{figure}[htb]
\begin{tikzpicture}[scale=0.5]
\coordinate (z1) at (1,0.5);
\coordinate (z2) at (1,-1.5);
\coordinate (z3) at (2,2.5) ;
\coordinate (z4) at (1.5,-1) ;
\coordinate (h) at (0,-0.7);
\coordinate (h0) at (0.7,0);
\coordinate (A) at (0,0);
\coordinate (B) at (14,0);
\coordinate (C) at (0,-11);
\coordinate (D) at (14,-11);
\coordinate (E) at (0,-22);

\fill [fill=gray!20] (0,0)--++(3.5,0) --++(z1) --++ (h0) --++ (z2) --++(2.8,0) --++(0,2.8)--++(-9,0)--cycle;

\draw (0,0)--++(3.5,0) node[midway,above] {$l_1$} node {\tiny $\bullet$}		 
                     --++ (z1) node[midway,above] {\tiny $z_1$} node {\tiny $\bullet$}
                       --++ (h0)  node[midway,above] {\tiny $h$} node {\tiny $\bullet$}
                     --++ (z2) node[midway,left] {\tiny $z_2$} node {\tiny $\circ$}
                     --++(2.8,0) node[midway, above] {$l_2$};

\fill [fill=gray!20] (0,-4)--++(3.5,0) --++ (z2) --++ (z3) --++(2.5,0) --++(0,-4)--++(-9,0)--cycle;

\draw (0,-4)--++(3.5,0)node[midway,below] {$l_1$} node {\tiny $\bullet$}
                     --++ (z2) node[midway,below left] {\tiny $z_2$} node {\tiny $\circ$}
                     --++ (z3) node[midway,below right] {\tiny $z_3$} node {\tiny $\circ$}
                     --++(2.5,0) node[midway, below] {$l_2$};

\draw [->]  (11,-2)--++(1,0);

\fill [fill=gray!20] (B)--++(3.5,0) --++(z1) --++ (h) --++ (z2) --++(3.5,0) --++(0,3.5)--++(-9,0)--cycle;

\draw (B)--++(3.5,0) node[midway,above] {$l_1$} node {\tiny $\bullet$}		 
                     --++ (z1) node[midway,above] {\tiny $z_1$} node {\tiny $\bullet$}
                       --++ (h)  node[midway] {\tiny $h\ \ $} node {\tiny $\bullet$}
                     --++ (z2) node[midway,left] {\tiny $z_2$} node {\tiny $\circ$}
                     --++(3.5,0) node[midway, above] {$l_2$};
                     
\draw ($(B)+(4.5,0.5)$)--++ ($(z2)+(h)$);

\fill [fill=gray!20] ($(B)+(0,-4)$)--++(3.5,0) --++ (z2) --++ (z3) --++(2.5,0) --++(0,-4)--++(-9,0)--cycle;

\draw ($(B)+(0,-4)$)--++(3.5,0)node[midway,below] {$l_1$} node {\tiny $\bullet$}
                     --++ (z2) node[midway,below left] {\tiny $z_2$} node {\tiny $\circ$}
                     --++ (z3) node[midway,below right] {\tiny $z_3$} node {\tiny $\circ$}
                     --++(2.5,0) node[midway, below] {$l_2$};
                     
\draw [->]  (-2,-13)--++(1,0);

\fill [fill=gray!20] (C)--++(3.5,0) --++(z1) --++ ($(h) +(z2)$) --++(3.5,0) --++(0,3.5)--++(-9,0)--cycle;

\draw (C)--++(3.5,0) node[midway,above] {$l_1$} node {\tiny $\bullet$}		 
                     --++ (z1) node[midway,above] {\tiny $z_1$} node {\tiny $\bullet$}                      
                     --++ ($(h) +(z2)$)  node[midway,left] {\tiny $z_2'$} node {\tiny $\circ$}
                     --++(3.5,0) node[midway, above] {$l_2$};

\fill [fill=gray!20] ($(C)+(0,-4)$)--++(3.5,0) --++ ($-1*(h)$) --++ ($(h)+(z2)$) --++ (z3) --++(2.5,0) --++(0,-4)--++(-9,0)--cycle;
\draw ($(C)+(3.5,-4)$)--++ ($(z2)$); 

\draw ($(C)+(0,-4)$)--++(3.5,0) node[midway,below] {$l_1$} node {\tiny $\bullet$}
		    --++ ($-1*(h)$) node[midway,left] {\tiny $h$} node {\tiny $\bullet$}
                     --++ ($(h)+(z2)$) node[midway,above] {\tiny $\ z_2'$} node {\tiny $\circ$}
                     --++ (z3) node[midway,below right] {\tiny $z_3$} node {\tiny $\circ$}
                     --++(2.5,0) node[midway, below] {$l_2$};

\draw [->]  (11,-13)--++(1,0);

\fill [fill=gray!20] (D)--++(3.5,0) --++(z1) --++ ($(h) +(z2)$) --++(3.5,0) --++(0,3.5)--++(-9,0)--cycle;

\draw (D)--++(3.5,0) node[midway,above] {$l_1$} node {\tiny $\bullet$}		 
                     --++ (z1) node[midway,above] {\tiny $z_1$} node {\tiny $\bullet$}                      
                     --++ ($(h) +(z2)$)  node[midway,left] {\tiny $z_2'$} node {\tiny $\circ$}
                     --++(3.5,0) node[midway, above] {$l_2$};

\fill [fill=gray!20] ($(D)+(0,-4)$)--++(3.5,0) --++ (h0) --++ ($(h)+(z2)$) --++ (z3) --++(1.8,0) --++(0,-3.3)--++(-9,0)--cycle;

\draw ($(D)+(0,-4)$)--++(3.5,0) node[midway,below] {$l_1$} node {\tiny $\bullet$}
		    --++ (h0) node[midway,below] {\tiny $h$} node {\tiny $\bullet$}
                     --++ ($(h)+(z2)$) node[midway,above ] {\tiny $\ z_2'$} node {\tiny $\circ$}
                     --++ (z3) node[midway,below right] {\tiny $z_3$} node {\tiny $\circ$}
                     --++(1.8,0) node[midway, below] {$l_2$};

\draw [->]  (-2,-24)--++(1,0);

\fill [fill=gray!20] (E)--++(3.5,0) --++(z1) --++ ($(z2)$) --++(3.5,0) --++(0,3.5)--++(-9,0)--cycle;

\draw (E)--++(3.5,0) node[midway,above] {$l_1$} node {\tiny $\bullet$}		 
                     --++ (z1) node[midway,above] {\tiny $z_1$} node {\tiny $\bullet$}                      
                     --++ ($(z2)$)  node[midway,left] {\tiny $z_2$} node {\tiny $\circ$}
                     --++(3.5,0) node[midway, above] {$l_2$};

\fill [fill=gray!20] ($(E)+(0,-4)$)--++(3.5,0) --++ (h0) --++ ($(z2)$) --++ (z3) --++(1.8,0) --++(0,-3.3)--++(-9,0)--cycle;

\draw ($(E)+(0,-4)$)--++(3.5,0) node[midway,below] {$l_1$} node {\tiny $\bullet$}
		    --++ (h0) node[midway,below] {\tiny $h$} node {\tiny $\bullet$}
                     --++ ($(z2)$) node[midway,below left ] {\tiny $\ z_2$} node {\tiny $\circ$}
                     --++ (z3) node[midway,below right] {\tiny $z_3$} node {\tiny $\circ$}
                     --++(1.8,0) node[midway, below] {$l_2$};      
                     
\end{tikzpicture}
\caption{Rotating (clockwise) the hole by $\pi$.}
\label{hole:rotating}
\end{figure}


Repeating this operation until the separatrix $l$ rotates  by the angle $(p+q+1)2\pi$ we get  a continuous family $S_{0,\theta}$, that we glue as in Section~\ref{connected:sum} to the surface $r_\theta.S_1$. For $\theta =(p+q+1)2\pi$, we have $r_\theta.S_1=S_1$, and by construction $S_{0,\theta }=S_{0,0}$. Hence, we get a closed path in the moduli space of meromorphic differentials. The corresponding transformation on the marked horizontal separatrices is $\tau_{P,Q}$.

\subsection{A realization of $\sigma_P$}\label{local:sigma}
As in Section~\ref{local:move}, let $P$ be a singularity of degree $p\neq \pm 1$, obtained after bubbling a handle as above, \emph{i.e.} we started from a singularity $P'$ of degree $p-2$ (with zero residue if $p<0$), and attached to it a torus in $\mathcal{H}(-p,p)$. A metric circle $\mathcal{C}$ around $P'$ is preserved by the construction and becomes a metric circle around the handle. Now we cut $S$ along $\mathcal{C}$ and  rotate the disc by an angle $\theta$.  We get a family of surfaces $(S_\theta)$. Since $\mathcal{C}$ is a $|p-2+1|$ covering of the Euclidean circle, for $\theta=2\pi(|p-2+1|)$, $S_\theta=S$. Keeping track of the marked horizontal separatrices, we see that the marked horizontal separatrix for $P$ has changed by an angle $-\theta =\pm 2\pi(p-1)$, 
hence $\pm 4\pi$ since the total angle at $P$ is $2\pi(p+1)$. As before, all the other horizontal separatrices are unchanged. Similarly, if $S_b^{hor}$ is in the same connected component as a surface where the singularity $P$ is obtained after the bubbling a handle, one gets the element $2\delta_P=\sigma_P$ of $Hor$.

 \subsection{Existence of the elementary moves} 
  The goal of this section is to prove Proposition~\ref{prop:g0} and Proposition~\ref{prop:g}. We first describe the cases when the hypothesis needed for  the transformations given in Sections~\ref{local:move}, \ref{nonlocal:move} and \ref{local:sigma}  are satisfied.
 
\begin{lemma}\label{lemme:tauPQ}
Let $S_b^{hor}$ be a framed surface of genus $g$ such that the underlying translation surface is in a nonhyperelliptic connected component $\mathcal{C}$ of a stratum $\mathcal{H}=\mathcal{H}(p,q,n_1,\dots ,n_r,m_1,\dots ,m_s)$, with $r,s\geq 0$, $n_1,\dots ,n_r>0$ and $m_1,\dots ,m_s<0$. Let $P,Q$ be singularities of $S_b^{hor}$ of degree $p,q$ respectively.
We assume that the conditions $(S)$ and $(CC)$ below are both realized:
\begin{enumerate}
\item[(S)] One of the following condition is fulfilled.
\begin{itemize}
\item[i)] $p,q\geq 0$,
\item[ii)] $p+q\leq -2$, and $\sum_j{m_j}\neq -1$ and ($g\neq 0$ or $r\neq 1$ or $s\neq 1$).
\item[iii)] $p+q\geq 0$, $p<0$ and $\sum_j m_j<-1$.
\end{itemize}
\item[(CC)] 
If $g=1$, we are not in the following case: $Rot(S_b)=d$ and $\mathcal{H}(p+q, n_1,\dots ,n_r,m_1,\dots ,m_s)=\mathcal{H}(-d,d)$ for some $d\geq 2$.
\end{enumerate}
Then, $\tau_{P,Q}\in Mon$.
\end{lemma} 


\begin{proof}
We denote by $\mathcal{H}_0$ the stratum $\mathcal{H}(p+q,n_1,\dots ,n_r,m_1,\dots ,m_s)$. 
Observe first that if $p,q\geq 0$ or $p+q\leq -2$, then the stratum $\mathcal{H}_0$ is not empty. If $p+q\geq 0$ and $p<0$, then the stratum $\mathcal{H}_0$ is empty if and only if $\sum_{i}m_i=1$. 

Also, remark that by Proposition~\ref{zero:residue}, Case $ii)$ corresponds to $p+q\leq -2$ and there exists in each connected component of $\mathcal{H}_0$ a flat surface with a pole of degree $p+q$ and zero residue.

Hence, the hypothesis $(S)$ implies that $\mathcal{H}_0$ is nonempty and we can break a singularity of degree $p+q$ of a surface $S_0\in \mathcal{H}_0$ and obtain an element $S\in\mathcal{H}$. We can also choose the connected component of $\mathcal{H}_0$ were $S_0$ is. From Sections~\ref{local:move} and~\ref{nonlocal:move}, $\tau_{P,Q}\in Mon$ as soon as we can obtain the connected component $\mathcal{C}$. 

\begin{itemize}
\item If the genus is zero, there is nothing more to prove since the stratum is connected.
\item If the genus is at least two, observe that breaking up a singularity preserves the parity of the spin structure when it is well defined, and $\mathcal{H}$ contains a component of a given parity if and only if  $\mathcal{H}_0$ also contains a component of the same parity (see Theorem~\ref{th:cc:g2}).

We claim that if $p\neq q$, then after breaking up the singularity we are not in a  hyperelliptic connected component. Indeed, in the local case, it is easy to see that if $S=S_0\oplus S_1$ is in the hyperelliptic component, then the hyperelliptic involution induces an involution on $S_0$ and $S_1$. But $S_1\in \mathcal{H}(-p-q-2,p,q)$, which is not a hyperelliptic component. In the nonlocal case, we see that the length of saddle connection corresponding to the small hole is unique (no other saddle connection has the same length), so the saddle connection is globally preserved by any isometry. Hence if $S$ is in a hyperelliptic connected component, it also induces an involution on the two pieces of surfaces, which is not possible.

When $p=q$, then $\mathcal{H}_0$ necessarily contains a nonhyperelliptic component (of the same parity as for $\mathcal{C}$), hence starting from a suitable surface and breaking up the singularity of order $p+q$ we obtain the required component $\mathcal{C}$.
\item Assume that the genus is one. Recall that in genus one, the components are classified by the rotation number: for a stratum $\mathcal{H}(k_1,\dots ,k_r)$, the rotation number can be any positive  common divisor of $k_1,\dots ,k_r$ except for a stratum of the type $\mathcal{H}(k,-k)$ were the rotation number must be different from $k$. Let $d$ be the rotation number corresponding to the component $\mathcal{C}$.  
Then $d$ divides $p+q,n_1,\dots ,m_s$. By hypothesis, $\mathcal{H}(p+q,n_1,\dots ,n_r)\neq \mathcal{H}(d,-d)$,  hence there exists a surface in $\mathcal{H}(p+q,n_1,\dots ,n_r)$ with rotation number $d$.  Breaking up the zero of degree $p+q$ into singularities of degree $p,q$ preserves the rotation number and gives the required surface. 
\end{itemize}
\end{proof}

\begin{lemma}\label{lemme:sigmaP}
Let $S_b^{hor}$ be a framed surface of genus $g\geq 1$ such that the underlying translation surface is in a nonhyperelliptic connected component $\mathcal{C}$ of a stratum $\mathcal{H}=\mathcal{H}(p,n_1,\dots ,n_r,m_1,\dots ,m_s)$, with $r,s\geq 0$ and the integers $n_1,\dots ,n_r>0$ and $m_1,\dots ,m_s<0$. Let $P$ be a singularity of degree $p$. We assume the conditions $(S)$ and $(CC)$ below are both realized:
\begin{enumerate}
\item[(S)] One of the following condition is fulfilled.
\begin{itemize}
\item[i)] $p\geq 2$,
\item[ii)] $p\leq -2$ and $\sum_j{m_j}\neq -1$ and ($g\neq 1$ or $r\neq 1$ or $s\neq 1$). 
\end{itemize}
\item[(CC)]
If $g=1$, we are not in the following case: $Rot(S_b)=|p|$ and $n_1,\dots ,n_r,m_1,\dots ,m_s$ are multiples of $p$.
\end{enumerate}
Then, $\sigma_P\in Mon$.
\end{lemma}

\begin{proof}
Denote by $\mathcal{H}_0$ the stratum $\mathcal{H}(p-2,n_1,\dots ,n_r,m_1,\dots ,m_r)$. As in the previous Lemma, condition $(S)$ imply that we can construct a surface in $\mathcal{H}$ by  bubbling a handle $S_1\in \mathcal{H}(-p,p)$ at a singularity of degree $p-2$ on a surface $S_0\in \mathcal{H}_0$. We denote by $S=S_0\oplus S_1\in \mathcal{H}$ the resulting surface. We must check that we can have  $S\in \mathcal{C}$. 

Observe that $S_0$ we can be in any connected component of $\mathcal{H}_0$, $S_1$ can be in any connected component of $\mathcal{H}(-p,p)$. Observe also that if $S_0\oplus S_1$ is in a hyperelliptic connected component, then necessarily $S_0$ and $S_1$ are in hyperelliptic connected components.
\medskip

We first assume that $g\geq 2$.

Assume that $p$ is odd then $\mathcal{H}(-p,p)$ does not contain a hyperelliptic component, hence $S$ is not in a hyperelliptic connected component of $\mathcal{H}$. Furthermore, by Theorem~\ref{th:cc:g2}, the stratum $\mathcal{H}$ has only one nonhyperelliptic connected component, therefore $S \in \mathcal{C}$.

 Assume that $p$ is even and positive.  If $p\geq 6$, the stratum $\mathcal{H}(-p,p)$ has two nonhyperelliptic components (one for each parity of the spin structure). Hence, for any choice of $S_0$, we can obtain any nonhyperelliptic component with even or odd spin structure. If $p=4$, the stratum $\mathcal{H}(-4,4)$  has two components: a nonhyperelliptic one, which has even spin structure (the rotation number is one), and the hyperelliptic one, with odd spin structure (the rotation  number is two). If there exists in $\mathcal{H}_0$ a nonhyperelliptic component, we use it and we obtain $S$ in the required components of $\mathcal{H}$. Otherwise, $\mathcal{H}_0=\mathcal{H}(2,-2)$ or $\mathcal{H}_0=\mathcal{H}(2,-1,-1)$, so $\mathcal{H}=\mathcal{H}(4,-2)$ or $\mathcal{H}=\mathcal{H}(4,-1,-1)$, which has only one nonhyperelliptic component. If $p=2$ then $\mathcal{H}_0= \mathcal{H}(0,n_1,\dots ,n_r,m_1,\dots ,m_s)$ is a stratum with no hyperelliptic connected component, hence 
 $S$ cannot be in a hyperelliptic component, and the parity of its spin structure is given by that of $S_0$, which can be odd or even.
 
The case $p$ even and negative is analogous and left to the reader.
\medskip

Now we assume that $g=1$.  
As in the proof of the previous lemma, let $d>0$ that divides $p,n_1,\dots ,m_s$.
We have $d\neq \pm p$, otherwise $n_1,\dots ,m_s$ are multiples of $p$ and this case is excluded. Hence, there exists a surface $S_1\in \mathcal{H}(-p,p)$ with $Rot(S_1)=d$, and therefore $Rot(S)=d$. 
 \end{proof}

 
 Now we can prove Proposition~\ref{prop:g0}  and Proposition~\ref{prop:g}.

 \begin{proof}[Proof of Proposition~\ref{prop:g0}]
We define the element $\rho=\sum_P \delta_P$, where the sum is taken on all singularities that are not simple poles. Observe that  rotating the surface by $2\pi$ and keeping track of the horizontal separatrices gives the element $\rho$ which is therefore in $Mon$.\\
The case with three singularities is special. Here $\mathcal{H}=\mathcal{H}(p,q,r)$, and denote by $P,Q,R$ the three singularities. Recall that  $(r+1)\delta_R=(p+1)\delta_P=(q+1)\delta_Q=0$ (see Section~\ref{mod:space:framed}).
 Since $g=0$ we also have,  $r+1=-1-p-q$, so
$$(r+1)\rho=-(r+1)(\delta_P+\delta_Q)=(p+q+1)(\delta_P+\delta_Q)=\tau_{P,Q}\in Mon.$$

We look at the cases where the hypothesis $(S)$ given in Lemma~\ref{lemme:tauPQ} is not satisfied (the genus being zero, the hypothesis $(CC)$ is satisfied).
\begin{itemize}
\item[a)] $p+q=-1$. Here, $\tau_{P,Q}=0$ so there is nothing to prove.
\item[b)]  $p+q\geq 0$ and $p<0$ and  $\sum_{j}m_j\geq -1$. This case does not appear since $p+q+\sum_{i}n_i+\sum_{j}m_j=-2$. 
\item[c)]  $p+q\leq -2$, and $r=s=1$. Denote respectively by $M,N$ the two other singularities, and respectively by $m<0$ and $n>0$ their degree. Observe that if $m+n\geq 0$, then from Case  b) above,  the hypothesis $(S)$ is necessarily satisfied. 
 We have  $n+m=-2-(p+q)\geq 0$ so $\tau_{M,N}\in Mon$. As before, the element $\rho=\delta_P+\delta_Q+\delta_M+\delta_N$  is in $Mon$.  Then, the condition $p+q+m+n=-2$ implies
$\tau_{M,N}-(n+m+1)\rho=-\tau_{P,Q}$, so $\tau_{P,Q}\in Mon$.

\item[d)] $p+q\leq -2$, and the stratum is of the form $\mathcal{H}(p,q,n_1,\dots ,n_r,-1)$ with $r\geq 1$ and $n_1,\dots ,n_r>0$. Denote by $M$ the simple pole, and by $N_1,\dots ,N_r$ the singularities of degree respectively $n_1,\dots ,n_r$. 
As before, we have $\delta_{N_k}=-\tau_{M,N_k}\in Mon$. Also, 
$$\tau_{P,Q}=(p+q+1)(\delta_P+\delta_Q)=(p+q+1)(\rho -\sum_k{\delta_{N_k}}).$$
Hence, $\tau_{P,Q}\in Mon$.
\end{itemize}

 \end{proof}

 \begin{proof}[Proof of Proposition~\ref{prop:g}]
 
 
 \medskip
 As in the previous proof , we define the element $\rho=\sum_P \delta_P$, where the sum is taken on all singularities that are not simple poles. Recall that $\rho\in Mon$.

 We first look at the element $\tau_{P,Q}$. We study the cases where the hypothesis $(S)$ given in Lemma~\ref{lemme:tauPQ} is not satisfied.
 \begin{itemize}
 \item $p+q=-1$, there is nothing to prove.
 \item  $p+q\geq 0$ and $p<0$ and $\sum_{j}m_j=0$. Since we suppose that $P$ is not the only pole, this case does not appear.
 \item $p+q\geq 0$ and $p<0$ and $\sum_{j}m_j=-1$. It means that except for the pole $P$ of order $p$, there is a unique other pole $M$ which is simple. The case $p=-1$ does not appear by hypothesis of the proposition (it corresponds to the case where there are exactly  two poles that are simple). If $p<-1$, we denote by $N_1,\dots ,N_r$ the singularities of degree $n_1,\dots ,n_r$ respectively. We remark that $\delta_{P_i}=-\tau_{N_i,M}\in Mon$ by Lemma~\ref{lemme:tauPQ} (since there is a simple pole, the hypothesis $(CC)$ of the lemma is automatically satisfied). Hence we can conclude as in Case $d)$ of the previous proposition.
 \item $p+q\leq -2$, and the stratum is of the form $\mathcal{H}(p,q,n_1,\dots ,n_r,-1)$ with $n_k>0$. We construct $\tau_{P,Q}$ as in Case $d)$ of the previous proposition.
 \end{itemize}
Hence there remains the case where $g=1$, the rotation number is $d$, and $\mathcal{H}(p+q,n_1,\dots ,n_r,m_1,\dots ,m_s)=\mathcal{H}(d,-d)$. 
\begin{itemize}
\item  $p+q=d$, hence the stratum is $\mathcal{H}(kd, (1-k)d,-d)$ for some $k>1$ with $p=kd$ and $q=(1-k)d$.
\item  $p+q=-d$, hence the stratum is $\mathcal{H}(kd, -(1+k)d,d)$ for some $k>0$ with $p=kd$ and $q=-(1+k)d$.
\end{itemize}
We postpone these two cases to the end of the proof. In the remaining of the proof we can use that $\tau_{P,Q}\in Mon$ in any cases except these two cases above.\medskip

Now we look at the element $\sigma_P$. We study the cases where the hypothesis given in Lemma~\ref{lemme:sigmaP} are not satisfied. 
\begin{itemize}
\item  $p\leq -2$ and $\sum_{j} m_j=-1$. There is exactly one other pole $M$ which is simple. Here, $\tau_{P,M}\in Mon$ and we have $\sigma_{P}=2\tau_{P,M}\in Mon$.

\item $p\leq -2$,  and $g=r=s=1$. There is exactly one other pole $M$ and one zero $N$ of degree $m,n$ respectively.  We have $m+n+p=0$.
If $\tau_{M,N}=(m+n+1)(\delta_M+\delta_N)\in Mon$, then 
$$(m+n+1)\rho-\tau_{M,N}=(-p+1)\delta_P=(-p-1+2)\delta_P=\sigma_P\in Mon.$$
Otherwise, from the above study, we have $n=-kp$, $m=(k-1)p$ for some $k\geq 2$ and the rotation number is $|p|$. The case $k=2$ corresponds to the hyperelliptic connected component of $\mathcal{H}(p,-2p,p)$ and therefore does not occur by hypothesis. Hence $k>2$ and therefore $\tau_{P,N}\in Mon$ (we are not in the two exceptional cases of the above study). Also, $\tau_{P,M}\in Mon $ for the same reasons. Hence
$$p\rho-\tau_{P,N}-\tau_{P,M}=(p-n-m)\delta_P=2p\delta_P=
-\sigma_P\in Mon.$$

\item The hypothesis $(CC)$ fails:  the stratum is $\mathcal{H}(\pm d, k_1d,\dots ,k_rd)$, and $\deg(P)=p=\pm d$. In this case, as before, we produce $\sigma_P$ as a combination of ``$\tau$'' elements.  If there is a singularity $Q$ of degree $q=-p$, then observe that $\tau_{P,Q}\in Mon$, hence 
\begin{eqnarray*}
(1+q)\tau_{P,Q}&=&(1+q)q\delta_P+(1+q)q\delta_Q=(1-p)(-p)\delta_P\\
&=&2\delta_P=\sigma_P\in Mon.\end{eqnarray*}
If there are least two other singularities $P',P''$ of degree $p$, then we have $\tau_{P,P'}+\tau_{P,P''}-\tau_{P',P''}=\sigma_P\in Mon$.
So we can assume that there are at most two singularities of degree $p$ and no singularities of degree $-p$.
  
   If there are two singularities $P,P'$ of degree $p$. Observe that $\tau_{P,Q}\in Mon$ for each $Q\neq P$. Indeed from the previous study, this may be false only in $\mathcal{H}(d,d,-2d)$ and in $\mathcal{H}(-d,-d,2d)$. But in these cases, the connected component with rotation number $d$ is precisely the hyperelliptic connected component. Also, we have $\sigma_Q\in Mon$. This implies 
 \begin{eqnarray*}
 \sum_{Q\neq P,P'}\left( 2\tau_{P,Q}-p \sigma_Q\right)&=&2\left(\sum_{Q\neq P,P'} \deg(Q) \right)\delta_P \\
 &=& -4p\delta_P=4\delta_P=2\sigma_P\in Mon.
 \end{eqnarray*}
 Hence if $p$ is even, $\sigma_{P}\in Mon$.
 If $p$ is odd,  we have
 
  \begin{eqnarray*}
&&  \rho+\tau_{P,P'}+\sum_{Q\neq P,P'} \left(\tau_{P,Q}-\frac{p+1}{2}\sigma_Q\right) \\
 &=&  \sum_{Q\neq P,P'} \left(\delta_Q+\tau_{P,Q}-(p+1)\delta_Q\right)
 \\
 &=& \sum_{Q\neq P,P'} \deg(Q)\delta_P = -2p\delta_P=2\delta_P=\sigma_P\in Mon.
  \end{eqnarray*}

 If $P$ is the only singularity of degree $p$, we have 
  $$\sigma_{P}=\sum_{Q\neq P} (2\tau_{P,Q} - p \sigma_{Q})\in Mon.$$
 \end{itemize}
 Therefore, we have proven that $\sigma_P\in Mon$ in each case.\medskip
 
 Now, we come back to a stratum of the kind  $\mathcal{H}(kd, (1-k)d, -d)$ (for some $k>1$), and we look at the component of rotation number $d$. We want to produce $\tau_{P,Q}$, where $\deg(P)=kd$ and $\deg(Q)=(1-k)d$. Note that if $k=2$ we are in the hyperelliptic connected component of $\mathcal{H}(2d,-d,-d)$, hence this case does not occur by hypothesis. Denote by $R$ the other singularity. We have
 \begin{enumerate}
\item if $d$ is even, $\sigma_P=2\delta_P\in Mon$,  hence $\delta_P\in Mon$. Similarly, $\delta_Q\in Mon$, so $\tau_{P,Q}\in Mon$.
\item if $d$ is odd, since $\tau_{P,Q}=(d+1)(\delta_P+\delta_Q)$, we get $\tau_{P,Q}=(d+1)\rho-\frac{d+1}{2}\sigma_R\in Mon$.
\end{enumerate}
The case with a stratum of the kind $\mathcal{H}(dn, -(1+k)d, d)$ is similar. This concludes the proof. 
 \end{proof}

\section{Positive genus}\label{pos:genus}
\subsection{A topological invariant} \label{topo:inv}
Here we describe a topological invariant for connected components of $\mathcal{H}^{hor}$ in the following cases:
\begin{itemize}
\item there are no simple poles, and there are singularities of odd degree.
\item there are exactly two poles that are simple, and some odd singularities of positive degree.
\end{itemize}

We first assume that there are no simple poles. The invariant is inspired by  the well known \emph{parity of spin structure} for translation surfaces with even degree singularities (\cite{KoZo}). See also \cite{B:labeled}.

For a  smooth closed curve $\gamma$ in $S$ that does not pass through any singularity, define $ind(\gamma)$ to be the index of the Gauss map defined by $\gamma'$. Choose $(\alpha_i,\beta_i)_{i\in \{1,\dots, g\}}$ a collection of smooth  simple closed curves representing a sympletic basis for the homology of $S$, and define 
$$\phi(\alpha,\beta)=\sum_{i=1}^g (ind(\alpha_i)+1)(ind(\beta_i)+1) \mod 2.$$
When $S$ has no odd degree singularities, $\phi(\alpha,\beta)$ does not depend on the choices of $(\alpha,\beta)$ and is the parity of the spin structure of $S$ (see \cite{B:merom,KoZo}).

When there are odd degree singularities, $\phi(\alpha,\beta)$ clearly depends on the choice of $(\alpha,\beta)$: indeed, if we continuously deform an element $\alpha_i$ or $\beta_i$ until we ``cross an odd singularity'', its index changes by a odd value.

Now we choose once for all an ordered pairing of the set of odd degree singularities, \emph{i.e.} we denote by $(P_1^-,P_1^+), \dots ,(P_s^-, P_s^+)$ these singularities. For a simple curve $\gamma$ joining $P_j^-$ to $P_j^+$, we define $ind(\gamma)$ to be the index (mod 2) of the Gauss map defined by a simple smooth path $\tilde{\gamma}$, whose image is in a small neighborhood of the image of $\gamma$, and such that:
\begin{itemize}
\item $\tilde{\gamma}$ is tangent in its starting point to the fixed horizontal separatrix of $P_j^-$
\item $\tilde{\gamma}$ is tangent in its ending point to the fixed horizontal separatrix of $P_j^+$, rotated by $\pi$.
\end{itemize}
Since $P_j^{+},P_j^-$ are both of odd degree, their corresponding conical angles are an even multiple of $2\pi$, and hence $ind(\gamma)$ does not depend on the choice of $\tilde{\gamma}$.

Now, for a fixed choice of $(\alpha_i,\beta_i)_i$, let $\gamma_1,\dots ,\gamma_s$ be a collection of simple curves, with no pairwise intersections, with $\gamma_j$ joining $P_j^-$ to $P_j^+$, and each $\gamma_j$  does not intersect  the $(\alpha_i,\beta_i)_i$. Then, we define 
$$Sp(\alpha,\beta,\gamma)=\phi(\alpha,\beta)+\sum_{j} ind(\gamma_j) \mod 2.$$

It is obvious that $Sp(\alpha,\beta,\gamma)$ can take two values, for different choices of horizontal separatrices. 
We will prove that $Sp(\alpha,\beta,\gamma)$ does not depend on the choice of $\alpha,\beta,\gamma$ (only on the choice of the oriented pairing of the odd degree singularities). Hence, $Sp$ defines a topological invariant for the connected components of $\mathcal{H}^{hor}$.

\begin{lemma}
$Sp(\alpha,\beta,\gamma)$ does not depend on the choice of $\gamma$.
\end{lemma}
\begin{proof}
Let $\gamma,\gamma'$ be two collections of simple curves as above. Up to making a small perturbation of $\gamma$ and $\gamma'$, we can assume that the number of intersection points of any two curves in this collection is finite.
The surface $D=S\backslash \cup_{i} (\alpha_i \cup \beta_i)$ is a topological disc with $g-1$ holes. By definition $\gamma_1$ and $\gamma_1'$ have the same end points. If they do not intersect in their interior, then $ind(\gamma_1)=ind(\gamma_1')+k$, where $k$ is the number of odd singularity
of a component of $D\backslash (\gamma_1\cup \gamma_1')$. In 
  this case, the number of intersection points (mod 2) between $\gamma_1'$ and the $(\gamma_{j})_{j\neq 1}$  is $k$. If $\gamma_1$ and $\gamma_1'$ have $N>0$ intersection points we find $\gamma_1''$ with the same endpoints as $\gamma_1$, such that $\gamma_1$ and $\gamma_1''$ have no interior intersection point and such that $\gamma_1''$ and $\gamma_1'$ have $N'<N$ intersection point and we iterate the procedure.

In particular replacing $\gamma_1$ by $\gamma_1'$ preserves the value:
$$\sum_{i} ind(\gamma_i)+ N(\gamma)  \mod 2.$$
where $N(\gamma)$ is the number of self-intersections of the family $\gamma$.

Hence, successively replacing $\gamma_i$ by $\gamma_i'$, we obtain:
$$\sum_{i} ind(\gamma_i)=\sum_{i} ind(\gamma_i') \mod 2.$$
\end{proof}

\begin{lemma}\label{lemme:gamma}
$Sp(\alpha,\beta,\gamma)=Sp(\alpha,\beta)$ does not depend on the choice of the symplectic basis $\alpha,\beta$.
\end{lemma}

\begin{proof}
Let $(\alpha,\beta,\gamma)$ and $(\alpha',\beta',\gamma')$ be two families of curves as above. We first show that, there exists $\alpha'', \beta''$ homotopic to $\alpha,\beta$, that do not intersect $\gamma'$ and such that:
$$Sp(\alpha,\beta,\gamma)=Sp(\alpha'',\beta'',\gamma').$$

By the previous lemma, we can choose $\gamma$ so that it does not intersect $\gamma'$. Let $\gamma_1'\in \gamma' $, and we assume that it intersects $\alpha, \beta$. Consider the last intersection point, \emph{i.e.} $x_0=\gamma_1'(t_0)$, and $\alpha,\beta$ do not intersect $(\gamma_1'(t))_{t>t_0}$. We assume for instance that the intersection is with $\alpha_1$.

\begin{figure}[htb]
\begin{center}
\labellist
\tiny 
\hair 2pt
\pinlabel $P_1^+$ at 400 530
\pinlabel $P_1^-$ at 170 520
\pinlabel $\alpha_1$ at 100 450
\pinlabel $\beta_1$ at 50 680
\pinlabel $\gamma_1'$ at 200 570
\pinlabel $\gamma_1$ at 450 580
\pinlabel $\alpha_1$ at 100 30
\pinlabel $\beta_1$ at 50 260
\pinlabel $\gamma_1'$ at 200 150
\pinlabel $\gamma_1$ at 450 185
\endlabellist
\includegraphics[scale=0.4]{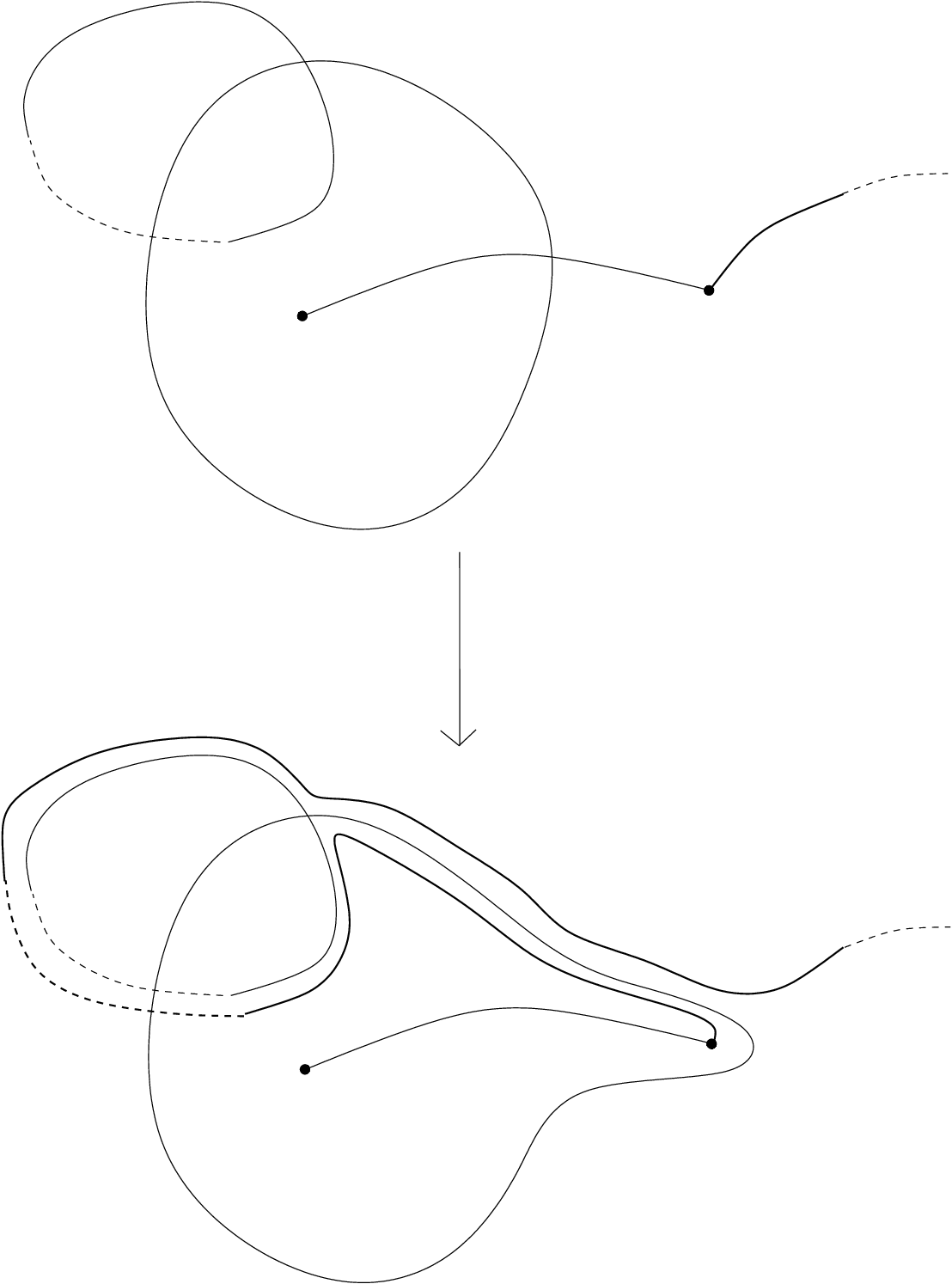}
\caption{Decreasing the number of intersection points.}
\label{lemme2sp}
\end{center}
\end{figure}

Now, we push $\alpha_1$ until it crosses the endpoint $P_1^+$. So, $ind(\alpha_1)$ is replaced by $ind(\alpha_1)\pm \deg(P_1^+)$. But now, $\alpha_1$ intersects $\gamma_1$.
For $\varepsilon$ small enough, the $\varepsilon$-boundary of $\alpha_1\cup \beta_1$, once removed $\alpha_1,\beta_1$ is a topological annulus, hence, we can modify $\gamma_1$ in that neighborhood to avoid $\alpha_1$  (see Figure~\ref{lemme2sp}). This replaces $ind(\gamma_1)$ by $ind(\gamma_1)+ind(\beta_1)+1$.
In particular $Sp(\alpha,\beta,\gamma)$ is not changed by this procedure, and the new family $(\alpha,\beta)$ has one intersection point less with $\gamma'$.

 Iterating the process, we eventually obtain $\alpha'',\beta''$ that do not intersect~$\gamma'$. 
 
Now, we consider the canonical continuous map $\phi:S\to \overline{S}$, where $\overline{S}$ is the surface obtained by collapsing each curve $\gamma_i'$ to a single point. The map $\phi$ induces an homeomorphism from $S\backslash \gamma'$ to its image. 

For a simple closed curve $\overline{c}=\phi(c)$ in $\overline{X}$, that does not pass through the image of a singularity, we define $ind({\overline{c}})=ind({c})$.  The map $\theta(\overline{c})=ind(\overline{c})+1 \mod 2$ defines a quadratic form on ${H}_1(\overline{S},\mathbb{Z}/2\mathbb{Z})$ (see \cite{Johnson,KoZo}). Hence its Arf invariant is
$$\sum_i (ind(\alpha_i'')+1)(ind(\beta_i'')+1)=\sum_i (ind(\alpha_i')+1)(ind(\beta_i')+1) \mod 2.$$
Hence $Sp(\alpha',\beta',\gamma')=Sp(\alpha'',\beta'',\gamma')=Sp(\alpha,\beta,\gamma)$.
\end{proof}

When the stratum is of the form $\mathcal{H}(-1,-1,n_1,\dots ,n_r)$, with positive integers $n_1,\dots ,n_r$,  we define the invariant after first cutting the two infinite cylinder, and gluing them together to make a finite cylinder, on a surface in the stratum $\mathcal{H}(n_1,\dots ,n_r)$. 

\begin{remark}\label{rem:function}
The invariant $Sp$ obtained depends on the choice of the  pairing $\{(P_1^-,P_1^+),(P_2^-,P_2^+),\dots ,(P_s^-, P_s^+)\}$ of the odd degree singularities. We can ask how $Sp(S)$ changes when we replace the pairing by another one. It is enough to study the case when we interchange $P_1^-$ with $P_1^+$ and when we interchange $P_1^-$ with $P_2^-$.
\begin{enumerate}
\item For the first case ($P_1^-$ with $P_1^+$). $Sp(S)$ is clearly replaced by $Sp(S)+1$.
\item For the second case, we replace again $Sp(S)$ by $Sp(S)+1$ . Indeed, consider as before two nonintersecting curves $\gamma_1,\gamma_2$ joining $P_1^-$ to $P_1^+$  and $P_2^-$ to $P_2^+$ respectively. Then, deform them until $\gamma_1,\gamma_2$ are tangent on a unique intersection point. Then, we obtain a new pair $\gamma_1',\gamma_2'$ joining $P_1^-$ to $P_2^+$  and $P_2^-$ to $P_1^+$ such that $Ind(\gamma_1)+Ind(\gamma_2)=Ind(\gamma_1')+Ind(\gamma_2')$. But $\gamma_1',\gamma_2'$ intersect (transversally) on one point. From the proof of Lemma~\ref{lemme:gamma}, modifying $\gamma_1',\gamma_2'$ so that they don't intersect will change the invariant by adding 1. 
\end{enumerate}
In particular, the invariant $Sp$ can be seen as a function from the set of pairings of odd degree singularities to $\mathbb{Z}/2\mathbb{Z}$, satisfying the above conditions. 
\end{remark}


\subsection{Proof of Theorem~\ref{MT:g}}
We assume first that there are only even degree singularities (or even degree zeroes, and a pair of simple poles.) We also assume that the underlying connected component is not a hyperelliptic one.

Let $P$ be a singularity of the base surface $S_p$. From Proposition~\ref{prop:g}, the element $\sigma_P=2\delta_P$ is in $Mon$. Since the singularity has even degree, $\delta_P\in Mon$. Hence, $Mon=Hor$.
\medskip

Now we assume that there are odd degree singularities. First observe that if there is a simple pole $P$ (except the case of two simple poles and no other poles), then for any $Q\neq P$, $\tau_{P,Q}=\delta_Q\in Mon$. Hence $Mon=Hor$.

So, we can assume that there are no simple poles. 
As before, for each singularity $P$ of even degree, we use $\sigma_P$ to see that $\delta_P\in Mon$. Now, fix a frame, and consider $P_1,\dots ,P_{2r}$ the singularities of odd degree. Then, for $i$ from 1 to $2r-1$ successively, we use $\sigma_{P_i}$ and $\tau_{P_i,P_{i+1}}$ to obtain an element of the form $\delta_{P_i}+k_i\delta_{P_i+1}\in Mon$. For, $P_{2r}$, we can only get half of possible horizontal separatrices, by using $\sigma_{P_{2r}}$. Hence, we see that $Mon$ is a subgroup of $Hor$ of index at most~2.  According to  Section~\ref{topo:inv}, in this case the invariant $Sp$ can take two values, therefore $Mon\neq Hor$ and $Mon$ is a subgroup of $Hor$ of index~2.

Observe that if there is only one pole $Q$, for a given singularity $P$, the element $\tau_{Q,P}$ is not necessarily possible. Hence, we first fix the separatrix of $Q$ by using $\rho=\sum_{R}\delta_R$, the element in $Mon$ that corresponds to rotating the surface by $2\pi$, and continue as above.

The case with two simple poles is similar and left to the reader.

\subsection{Hyperelliptic connected component}\label{hyp:case}
From \cite{B:merom}, a hyperelliptic connected component of the moduli space of meromorphic differentials is necessarily a component of a stratum of the following kind:
\begin{itemize}
\item $\mathcal{H}(n,n,p,p)$
\item $\mathcal{H}(2n,p,p)$
\item $\mathcal{H}(n,n,2p)$
\item $\mathcal{H}(2n,2p)$
\end{itemize}
for some, $n>0$ and $p<0$.

Let $\mathcal{C}_{hyp}$ be a hyperelliptic connected component of the moduli space of translation surface with poles. 
Let $\mathcal{H}_{\mathcal{C}_{hyp}}^{hor}$ be the set of framed translation surfaces whose underlying surfaces are in $\mathcal{C}_{hyp}$. We assume that there are two (marked) zeroes $N_1,N_2$ of the same degree. Denote by $i$ the hyperelliptic involution. Since $i(N_1)=N_2$, the image by  $i$ of the marked horizontal separatrix $l_1$ of $N_1$ is a horizontal separatrix $i(l_1)$ of $N_2$. The angle between the marked horizontal separatrix $l_2$ of $N_2$ and $i(l_1)$ is an odd multiple of $\pi$ and is between $\pi$ and $(2n+1)\pi$ and is invariant by continuous deformations. Hence, it is an invariant $\Phi_{zeroes}$ of connected components, which can clearly get $n+1$ values. Similarly, if there are two poles of the same degree, there is an analogous invariant  $\Phi_{poles}$ for the horizontal separatrices associated to the pair of poles, with $|p+1|$ values.

We have the following lemma:
\begin{lemma}
Let $\mathcal{H}_{\mathcal{C}_{hyp}}^{hor}$ be a hyperelliptic connected component with framed horizontal separatrices. Let $S_b^{hor}\in \mathcal{H}_{\mathcal{C}_{hyp}}^{hor}$. Let $P\in S_b^{hor}$ be a (marked) singularity. 
\begin{itemize}
\item If there exists another singularity $P'$ of the same degree, then $\tau_{P,P'}\in Mon$.
\item Otherwise,  $\sigma_P\in Mon$. 
\end{itemize}
\end{lemma}
\begin{proof}
The proof is easy and left to the reader.
\end{proof}

This lemma, associated to the definition of the invariant gives the following theorem.
\begin{theorem}
Let $\mathcal{H}_{\mathcal{C}_{hyp}}^{hor}$ be a hyperelliptic connected component with marked horizontal separatrices. 
\begin{itemize}
\item If $\mathcal{H}_{\mathcal{C}_{hyp}}^{hor}\subset \mathcal{H}(n,n,p,p)$, for some $n>0$ and $p<-1$, then $\mathcal{H}_{\mathcal{C}_{hyp}}^{hor}$ has $(n+1)|p+1|$ connected components distinguished by the maps $\Phi_{zeroes}$ and $\Phi_{poles}$.
\item If $\mathcal{H}_{\mathcal{C}_{hyp}}^{hor}\subset \mathcal{H}(n,n,2p)$, for some $n>0$ and $p<0$, or $\mathcal{H}_{\mathcal{C}_{hyp}}^{hor}\subset \mathcal{H}(n,n,-1,-1)$ then $\mathcal{H}_{\mathcal{C}_{hyp}}^{hor}$ has $(n+1)$ connected components distinguished by the map $\Phi_{zeroes}$.
\item If $\mathcal{H}_{\mathcal{C}_{hyp}}^{hor}\subset \mathcal{H}(2n,p,p)$, for some $n>0$ and $p<-1$, then $\mathcal{H}_{\mathcal{C}_{hyp}}^{hor}$ has $|p+1|$ connected components distinguished by the maps $\Phi_{poles}$.
\item If $\mathcal{H}_{\mathcal{C}_{hyp}}^{hor}\subset \mathcal{H}(2n,2p)$ for some $n>0$ and $p<-1$ or $\mathcal{H}_{\mathcal{C}_{hyp}}^{hor}\subset \mathcal{H}(2n,-1,-1)$, then $\mathcal{H}_{\mathcal{C}_{hyp}}^{hor}$ is connected.
\end{itemize}
\end{theorem}
\begin{proof}
The proof is easy and left to the reader.
\end{proof}

 \section{Zero genus}\label{zero:genus}
Let  $\mathcal{H}=\mathcal{H}(n_1,\dots n_r)$ be a stratum of genus zero translation surfaces.
In this section, we count the number of connected components of $\mathcal{H}^{{hor}}(n_1,\dots ,n_r)$ and define a topological invariant distinguishing these connected components.

We assume that  there are no simple poles. Then, for $i\neq j$, we denote by $N_{ij}$ the (positive) integer:
$$N_{ij}=\gcd\left(\{n_k\}_{k\notin \{i,j\}} \cup \{n_i+1,n_j+1\}\right).$$

Let $S\in \mathcal{H}^{hor}(n_1,\dots ,n_r)$, and denote $P_1,\dots ,P_r$ the (marked) singularities of degree $n_1,\dots ,n_r$ respectively. For any $i<j$, let $\gamma_{ij}$ be a path joining $P_i$ to $P_j$ according to the marked horizontal separatrices (as in Section~\ref{topo:inv}). Then $ind(\gamma_{ij})$ is an integer and $ind(\gamma_{ij})\mod N_{ij}$ does not depend on the choice of $\gamma_{ij}$ (only on the choice of marked directions).

Now we define $\Phi(S)$ as:
$$\Phi(S)=\left( ind(\gamma_{ij})\right)_{i<j}\in \prod_{i<j} \mathbb{Z}/N_{ij}\mathbb{Z}.$$

The map $\Phi$ is clearly a locally constant map, and hence, an invariant of connected components of $\mathcal{H}^{hor}(n_1,\dots ,n_r)$. Note that $\Phi$ depends implicitly on the choice of the ordering of the singularities.
%
Note that the map $Sp$ is also well defined if there are some odd degree singularities.

\begin{theorem}\label{th:genus0}
Let  $\mathcal{H}=\mathcal{H}(n_1,\dots n_r)$ be a stratum of genus zero translation surfaces.
\begin{itemize}
\item If there exists $i_0\in \{1,\dots ,r\}$ such that $n_{i_0}=-1$, then $\mathcal{H}^{hor}$ is connected.
\item If all $n_i$ are different from $-1$ and if there are at most two odd degree singularities, then there are $N=\prod_{i<j} N_{ij}$ connected components of $\mathcal{H}^{hor}$, and two elements $S_1$ and $S_2$ of $\mathcal{H}^{hor}$ are in the same connected component if and only if $\Phi(S_1)=\Phi(S_2)$.
\item Otherwise, there are $2N$ connected components  of $\mathcal{H}^{hor}$, and two elements $S_1$ and $S_2$ of $\mathcal{H}^{hor}$ are in the same connected component if and only if $\Phi(S_1)=\Phi(S_2)$ and $Sp(S_1)=Sp(S_2)$.
\end{itemize}
\end{theorem}

Note that the first part is obvious, since $\tau_{P_{i_0},Q}=\delta_Q$ is in $Mon$. So, from now, we assume that there are no simple poles.

\begin{lemma} \label{lem:Phi:surj}
The map $\Phi$ is surjective. Furthermore, if there are at least three odd degree singularities, the map $\Phi\times Sp$  is surjective.
\end{lemma}

\begin{proof}
We first prove that $\Phi$ is surjective. 
Let $i_0\in \{1,\dots ,r\}$. When we replace the marked horizontal separatrix $l_{i_0}$ corresponding to $P_{i_0}$ by the one obtained by rotating $l_{i_0}$ by $2\pi$ counterclockwise, it adds to $\Phi(S)$ the element $\eta_{i_0}$ whose value is
\begin{itemize}
\item $-1$  in the factor  $\mathbb{Z}/N_{i_0j}\mathbb{Z}$ for each $j>i_0$.
\item $1$  in the factor  $\mathbb{Z}/N_{ji_0}\mathbb{Z}$ for each $j<i_0$.
\item 0 in the other factors.
\end{itemize}
Since the integers $\{N_{ij}\}_{i< j}$ are pairwise relatively prime, the element $\eta_{i_0}$ generates $\prod_{j\neq i_0} \mathbb{Z}/N_{i_0j}\mathbb{Z}$. In particular, $\eta_1,\dots ,\eta_r$ generates the group $\prod_{i<j} \mathbb{Z}/N_{ij}\mathbb{Z}$. So the map $\Phi$ is surjective.

When there are at least three odd degree singularities, $N=\prod_{i<j} N_{ij}$ is odd. In particular, choosing $i$ so that $n_i$ is odd, and rotating $l_i$ by $2\pi\prod_{j\neq i} N_{ij}$ does not change $\Phi(S)$, but changes $Sp(S)$ by $1=\prod_{j\neq i} N_{ij}$ (mod 2), so the map $\Phi\times Sp$ is surjective.
\end{proof}

\begin{lemma}\label{elements:g0}
Let $k\in \{1,\dots ,r\}$, for each $i,j\neq k$, with $i\neq j$ the following elements are in the group $Mon$. 
\begin{itemize}
\item $n_i(n_i+1) \delta_{P_k}$,
\item $2n_i n_j \delta_{P_k}$.
\end{itemize}
Furthermore, the subgroup of $Mon$ generated by these elements, seen as a subgroup of $\mathbb{Z}/(n_k+1)\mathbb{Z}$, is the subgroup generated by $\varepsilon_k\prod_{i\neq k}N_{ki}$, where
$\varepsilon_k=2$ if $n_k$ is odd and there are at least two other odd singularities, and $\varepsilon_k=1$ otherwise.
\end{lemma}

\begin{proof}
A direct computation shows that the element $n_i(n_i+1) \delta_{P_k}$ is given by $(n_i+1)\tau_{P_k,P_i}$, and the element $2n_i n_j \delta_{P_k}$ is given by
$n_j \tau_{P_k,P_i}+n_i \tau_{P_k,P_j}-n_k\tau_{P_i,P_j}.$

The subgroup of $\mathbb{Z}/(n_k+1)\mathbb{Z}$ generated by these elements is $<d_k>$, where:
$$d_k= \gcd\left(\{2n_in_j\}_{i\neq j \neq k}\cup\{n_i(n_i+1)\}_{i\neq k}\cup \{n_k+1\}\right).$$

Let $p>2$ be a prime number that divides $d_k$, $\alpha=\nu_p(d_k)$ its p-adic valuation. By definition of $d_k$, $p^\alpha$ divides $n_k+1$, and for each $i\neq k$, $p^\alpha$ divides $n_i$ or $n_i+1$. It cannot always divides  $n_i$ since 
$(n_k+1)+\sum_{i\neq k} n_i=-1$. Also, if there are two indices $i\neq j$, with $i,j\neq k$ such that $p^\alpha$ does not divide $n_i$ and $n_j$, then $p^\alpha$ does not divide $2n_in_j$, which contradicts $p^\alpha|d_k$. Hence there is exactly one index $i\neq k$ such that $p^\alpha$ does not divide $n_i$, hence, divides $n_i+1$. So, $p^\alpha|N_{ki}$.

Conversely, let $p>2$ be a prime number and $\alpha=\nu_p(\varepsilon_k \prod_{i\neq k} N_{ki})$. Since the $\{N_{ki}\}_{i}$ are pairwise relatively prime, there is an index $i_0$ such that $p^\alpha|N_{ki_0}$. Hence, we easily see that $p^\alpha$ divides $d_k$.
\medskip

We have proven that $\nu_{p}(d_k)=\nu_p(\varepsilon_k \prod_{i\neq k} N_{ki} )$ for $p>2$. Now, we prove the same for the case $p=2$.
If $n_k$ is even, then both $d_k$ and $\varepsilon_k\prod_{i\neq k}N_{ki}$ are odd.

Assume that $n_k$ is odd and there are at least three odd singularities. Denote by $i_0,j_0\neq k$ the indices of two odd degree singularities different from $P_k$. We see that the 2-adic valuation of $d_k$ is 1 by considering $2n_{i_0}n_{j_0}$, and the 2-adic valuation of $\varepsilon_k \prod_{i\neq k} N_{ki}$ is also 1 since $\varepsilon_k=2$ and all the $N_{ki}$ are odd.

Assume that $n_k$ is odd and there are exactly two odd degree singularities $P_k,P_{i_0}$ on the surface.
Let $\alpha>0$ such that $2^\alpha|d_k$. Then $2^\alpha$ divides $n_k+1$, and $n_{i}(n_{i}+1)$ for each $i\neq k$. In particular, it divides $(n_{i_0}+1)$ (since $n_{i_0}$ is odd), and for each $i\notin\{k,i_0\}$, it divides $n_{i}$ (since $n_{i}+1$ is odd). So, $2^{\alpha}|N_{ki_0}$.

Conversely, let $\alpha>0$ such that $2^\alpha|\prod_{i\neq k} N_{ki}$. The integer $N_{ki}$ is even if and only if $i=i_0$. Hence, $2^{\alpha}|N_{ki_0}$ and $2^\alpha$ divides each $n_i$, for $i\notin\{k,i_0\}$. So, it divides $n_in_j$, for each $i,j\neq k$ with $i\neq j$. Finally, $2^\alpha|d_k$.

Hence, we have proven that $d_k=\varepsilon_k \prod_{i\neq k} N_{ki}$.
\end{proof}

\begin{proof}[Proof of Theorem~\ref{th:genus0}]
We first assume that there are at most two odd degree singularities, so that for each $i$, $\varepsilon_i=1$  in the above lemma. In order to simplify notation, we set 
$N_{ii}=1$ for each $i$.

From Lemma~\ref{lem:Phi:surj}, $Mon$ has at most 
$$N=\frac{\prod_{i\in\{1,\dots ,r\}}|n_i+1|}{\prod_{1\leq i<j \leq r}N_{ij}}.$$
elements. The theorem will follow if we prove that $Mon$ has exactly this number of elements.

Since for each $i,j$, $N_{ij}|(n_i+1)$, there is a canonical map $ \mathbb{Z}/{(n_i+1)\mathbb{Z}}\to \mathbb{Z}/N_{ij}\mathbb{Z}$, so it induces a canonical map:
$$\Psi:
\begin{array}{ccc}
\prod_{i=1}^r \mathbb{Z}/{(n_i+1)}\mathbb{Z} &\to& \prod_{i,j\in\{1,\dots ,r\}} \mathbb{Z}/N_{ij}\mathbb{Z} \\
(x_i)_{i} &\mapsto&  (x_i \mod N_{ij})_{i,j}
\end{array}.
$$
Since $\{N_{ij}\}_{i,j}$ are relatively prime, by the Chinese Remainder Lemma, the kernel of $\Psi$ is $\prod_{i} u_i\mathbb{Z}/(n_i+1)\mathbb{Z}$, where $u_i=\prod_{j\neq i} N_{ij}$ and is a subgroup of $Mon$ by the previous lemma.

For a pair $(i_0,j_0)$ of distinct indices, the image by $\Psi$ of the element $-\tau_{P_{i_0},P_{j_0}}$ is the element $E_{i_0j_0}+E_{j_0i_0}$, where $E_{ij}$ is the element which is 1 for the indices $(i,j)$ and 0 everywhere else.

In particular, the image $\Phi(Mon)$ contains at least $\prod_{i<j} N_{ij}$ elements. So $Mon$ contains at least,
$$\frac{\prod_{i=1}^r|n_i+1|}{\prod_{i=1}^r u_i} \prod_{1\leq i<j\leq r}N_{ij}=\frac{\prod_{i=1}^r|n_i+1|}{\prod_{1\leq i<j\leq r}N_{ij}} $$
elements. Therefore it contains exactly that many elements.

Now we assume that there are at least three odd degree singularities. In order to simplify the notation, we define, for $i\neq 0$ $N_{i0}=N_{0i}=\varepsilon_i$, where, $\varepsilon_i=2$ if $n_i$ is odd and $\varepsilon_i=1$ otherwise.
From Lemma~\ref{lem:Phi:surj}, $Mon$ has at most 
$$N'=\frac{1}{2}\frac{\prod_{i=1}^r|n_i+1|}{\prod_{1\leq i<j \leq r}N_{ij}}$$
elements. We proceed as before, but replace the map $\Psi$ by the map $\tilde{\Psi}$
$$\tilde{\Psi}:
\begin{array}{ccc}
\prod_{i=1}^r \mathbb{Z}/{(n_i+1)}\mathbb{Z} &\to& \prod_{i=1}^r 
\prod_{j=0}^r \mathbb{Z}/N_{ij}\mathbb{Z}
\\
(x_i)_{i}& \mapsto& (x_i \mod N_{ij})_{(i,j)\in \{1,\dots ,r\}\times\{0,\dots ,r\}}
\end{array}.
$$

Since all $N_{ij}$ are odd and pairwise relatively prime, we see as before that the kernel is $\prod_{i} u_i\mathbb{Z}/(n_i+1)\mathbb{Z}$, where $u_i=\prod_{j=0}^r N_{ij}=\varepsilon_i \prod_{j\neq i} N_{ij}$,  and is a subgroup of $Mon$ by the previous lemma.

If $n_{i_0}$ or $n_{j_0}$ is even, then  the image by $\tilde{\Psi}$ of the element $-\tau_{P_{i_0},P_{j_0}}$ is the element $E_{i_0j_0}+E_{j_0i_0}$. If both $n_{i_0}$ and $n_{j_0}$ are odd, then we get
$E_{i_0j_0}+E_{j_0i_0}+E_{i_00}+E_{j_00}$. Therefore $\tilde{\Psi}((n_{i_0}+n_{j_0})\tau_{P_{i_0},P_{j_0}})=2(E_{i_0j_0}+E_{j_0i_0})$. Since $N_{i_0j_0}$ is odd, there is a multiple of $\tau_{P_{i_0},P_{j_0}}$ whose image by $\tilde{\Psi}$ is $E_{i_0j_0}+E_{j_0i_0}$. Finally, we obtain that the image by $\tilde{\Psi}$ of $Mon$ contains at least
$2^{n-1}\prod_{i<j}N_{ij}$ elements, where $n$ is the number of odd degree singularities, and therefore, $Mon$ has at least $N'$ elements. Hence it contains exactly that many elements.
\end{proof}

\section{Partially marked surfaces}
Coming back to the initial motivation of this paper. It is natural to study the moduli space of surfaces where only a subset of the singularities have a marked horizontal separatrix. 

Let $g\geq 1$. Let $\mathcal{H}(n_1^{\alpha_1},\dots ,n_r^{\alpha_r})$ be a stratum of the moduli space of genus $g$ meromorphic differentials, and let $\mathcal{C}\subset\mathcal{H}(n_1^{\alpha_1},\dots ,n_r^{\alpha_r})$ be a nonhyperelliptic connected component. Let $\{n_1^{\beta_1},\dots ,n_r^{\beta_r}\}\subset \{n_1^{\alpha_1},\dots ,n_r^{\alpha_r}\}$, and let $\mathcal{H}_\mathcal{C}^{part}$ be the corresponding moduli space of partially framed surfaces.

The following result follows easily from Theorem~\ref{MT:g}. 
\begin{corollary}
Assume that there are non-marked odd degree singularities, then $\mathcal{H}_\mathcal{C}^{part}$ is connected. Otherwise, $\mathcal{H}_\mathcal{C}^{part}$ has the same number of connected components as $\mathcal{H}_\mathcal{C}^{hor}$.
\end{corollary}

Assume now that the genus is zero, and denote by $\{P_1,\dots ,P_r\}$ the singularities, with $\{P_1,\dots ,P_s\}$, $s<r$ the marked ones. 
Now we define $\Phi_{part}(S)$ as:
$$\Phi_{part}(S)=\left( ind(\gamma_{ij})\right)_{1\leq i<j\leq s}\in \prod_{1\leq i<j \leq s} \mathbb{Z}/N_{ij}\mathbb{Z}.$$
\emph{i.e.} we restrict the map $\Phi$ to the marked singularities.

The following Theorem is an easy corollary of Theorem~\ref{th:genus0}.
\begin{corollary}
Let  $\mathcal{H}=\mathcal{H}(n_1,\dots n_r)$ be a stratum of genus zero translation surfaces, such that $\mathcal{H}^{hor}$ is not connected.
\begin{itemize}
\item If there are some non-marked odd degree singularities, or if there are at most two odd degree singularities, then two elements $S_1$ and $S_2$ of $\mathcal{H}^{part}$ are in the same connected component if and only if $\Phi_{part}(S_1)=\Phi_{part}(S_2)$.
\item Otherwise, two elements $S_1$ and $S_2$ of $\mathcal{H}^{hor}$ are in the same connected component if and only if $\Phi_{part}(S_1)=\Phi_{part}(S_2)$ and $Sp(S_1)=Sp(S_2)$.
\end{itemize}
\end{corollary}

\appendix
\section{More about connected sums}
We look back at the construction described in Section~\ref{connected:sum}. 
Let $S, S'$ be translation surfaces, and let $N\in S$, be a singularity of degree $n\geq 0$ and let $N'\in S'$ be a singularity of degree $n'=-2-n<0$ with zero residue.

First, we observe that once the scaling of $S'$, the neighborhood $U$ of $N$ and the pointed neighborhood $V$ of $N'$ are fixed, there remains a finite number of possibilities for gluing together $S\backslash U$ and $S'\backslash V$. There are exactly $n+1$ choices. If $S,S'$ are translation surfaces with marked horizontal separatrices, we can fix this choice by imposing that the separatrix corresponding to $N$ coincides with the separatrix corresponding to $N'$ (rotated by $\pi$). Once this combinatorial choice is fixed, the other choices involved in the construction (scaling, $U$, $V$) give a connected set of surfaces.

One would like to glue $S$ and $S'$ along several pairs $(N_i,N'_i)\in S\times S'$. We assume that for each $i$ the singularity $N_i$ has degree $n_i>0$, and the singularity $N'_i$ has degree $-2-n_i<0$ with zero residue. It is natural to glue successively along $(N_1,N'_1)$ then $(N_2,N'_2)$ and so on. However, after the first step, the singularities belong to the same surface. self-gluing construction is made analogously, but in general it is not possible any more to shrink one side to make space. However, in this case, shrinking sufficiently $S'$ at first solves this issue (since all singularities $N'_i$ have negative degree). As before, there is a combinatorial choice for each pair which is fixed if the initial surfaces are with marked horizontal separatrices, and two gluings with the same combinatorial choices give surfaces in the same connected component. 

Computing the $Sp$ invariant of the new surface is easy. It is enough to consider simple gluings and self-gluings. 
\begin{enumerate}
\item For simple gluing, there are two cases: either the two singularities are even, or they are odd. In the first case, the Sp invariant of the new surface is clearly the sum of the Sp invariant of the two surfaces. 
In the second case, following Remark~\ref{rem:function} we first choose a suitable pairing of odd degree singularities: we consider pairings of the form $\{(P_1^-,P_1^+)\}\cup \mathcal{P}$ with $P_1^+=N$ for $S$ and $\{(P_1'^-,P_1'^+)\}\cup \mathcal{P}'$ with $P_1'^-=N'$. We consider the following pairing for the new surface
$$\{(P_1^-,P_1'^+)\}\cup \mathcal{P}\cup \mathcal{P}'.$$
Then, we easily see that, with these pairings, the Sp invariant of the new surface is the sum of the Sp invariant of $S$ and $S'$.

\item For self-gluing, the new surface has genus one more than the initial surface. We easily see that the Sp invariant does not change for any pairing when glued singularities are of even order, and for a pairing containing $(N,N')=(P_1^-,P_1^+)$ when the glued singularities have odd order. 
\end{enumerate}

\nocite*
\bibliographystyle{plain}
\bibliography{biblio}

\end{document}